\documentclass[12pt]{article}
\usepackage[dvips]{graphicx}
\usepackage{amssymb,color,latexsym}
\usepackage[centertags]{amsmath}
\usepackage{amsfonts}

\def\BBox{\kern -0.2cm\hbox{\vrule width 0.2cm height 0.2cm}}

\newtheorem{theorem}{Theorem}
\newtheorem{lemma}[theorem]{Lemma}

\newcommand{\sst}[1]{${\scriptstyle #1}$}
\newcommand{\po}{\mathcal{P}}

\title{Symmetric Graphicahedra}
\author{
Maria Del R\'io-Francos\thanks{mad210fcos@gmail.com},
Isabel Hubard\thanks{Partially supported by Sociedad Matem\'atica Mexicana-Fundaci—n Sof'a Kovalevskaia and PAPIIT-M\'exico IB101412,  hubard@matem.unam.mx}, 
Deborah Oliveros\thanks{PAPIIT-M\'exico under project 104609, dolivero@matem.unam.mx}\\
{\small  Instituto de Matem\'{a}ticas}\\
{\small  Universidad Nacional Aut\'{o}noma de M\'{e}xico, M\'{e}xico}
\\[1ex]
Egon Schulte\thanks{Supported by NSF-Grant DMS--0856675, schulte@neu.edu}
\thanks{ PAPIIT-M\'exico under project 104609.}\\
{\small Department of Mathematics}\\
{\small Northeastern University, Boston, USA}\\ }

\begin{document}
\maketitle

\begin{abstract}
Given a connected graph $G$ with $p$ vertices and $q$ edges, the $G$-graphicahedron is a vertex-transitive simple abstract polytope of rank $q$ whose edge-graph is isomorphic to a Cayley graph of the symmetric group $S_p$ associated with $G$. The paper explores combinatorial symmetry properties of $G$-graphicahedra, focussing in particular on transitivity properties of their automorphism groups. We present a detailed analysis of the graphicahedra for the $q$-star graphs $K_{1,q}$ and the $q$-cycles $C_q$. The $C_q$-graphicahedron is intimately related to the geometry of the infinite Euclidean Coxeter group $\tilde{A}_{q-1}$ and can be viewed as an edge-transitive tessellation of the $(q-1)$-torus by $(q-1)$-dimensional permutahedra, obtained as a quotient, modulo the root lattice $A_{q-1}$, of the Voronoi tiling for the dual root lattice $A_{q-1}^*$ in Euclidean $(q-1)$-space.
\end{abstract}

\section{Introduction}

In the present paper, we continue the investigation of \cite{graphi} of an interesting class of abstract polytopes, called {\em graphicahedra\/}, which are associated with graphs. Any finite connected graph $G$ determines a $G$-graphicahedron. When $G$ has $p$ vertices and $q$ edges, this is a vertex-transitive simple abstract polytope of rank $q$ whose edge-graph ($1$-skeleton) is isomorphic to the Cayley graph of the symmetric group $S_p$ associated with a generating set of $S_p$ derived from $G$. The graphicahedron of a graph is a generalization of the well-known permutahedron, which is obtained when $G$ is a path. The study of the permutahedron has a rich history (see \cite{fom,post,Z95}), apparently beginning with Schoute~\cite{S11} in 1911; it was rediscovered in Guilbaud \& Rosenstiehl~\cite{guro} and named ``permuto\'{e}dre" (in French). 

In this paper, we explore combinatorial symmetry properties of graphicahedra, in particular focussing on graphicahedra with automorphism groups acting transitively on the faces of some given positive rank. We present a detailed analysis of the structure of the graphicahedra for the $q$-star graphs $K_{1,q}$ and the $q$-cycles $C_q$. 

The $K_{1,q}$-graphicahedon, with $q\geq 3$, is a regular polytope of rank $q$ with automorphism group $S_{q+1}\times S_q$, all of whose faces of a given rank $j$ are isomorphic to $K_{1,j}$-graphicahedra; in fact, the $K_{1,q}$-graphicahedron is the universal regular polytope with facets isomorphic to $K_{1,q-1}$-graphicahedra and vertex-figures isomorphic to $(q-1)$-simplices. The structure of the $C_q$-graphicahedron displays a beautiful connection with the geometry of the infinite Euclidean Coxeter group $\tilde{A}_{q-1}$ (see \cite[Ch. 3B]{arp}), whose Coxeter diagram is the $q$-cycle $C_q$. The $C_q$-graphicahedron, with $q\geq 3$, can be viewed as an edge-transitive tessellation of the $(q-1)$-torus by $(q-1)$-dimensional permutahedra, obtained as a quotient, modulo the root lattice $A_{q-1}$, of the Voronoi tiling for the dual root lattice $A_{q-1}^*$, in Euclidean $(q-1)$-space.

The paper is organized as follows. In Section~\ref{dron}, we review basic definitions for polytopes and graphicahedra, and describe face-transitivity properties of $G$-graphicahedra in terms of subgraph transitivity properties of the underlying graph~$G$. Then Sections~\ref{k1qgraphi} and~\ref{cqgraphi} deal with the graphicahedra for the $q$-star $K_{1,q}$ and the $q$-cycle $C_q$.

\section{Face-transitive graphicahedra}
\label{dron}

An {\it abstract polytope} of rank $n$, or an {\it $n$-polytope}, is a partially ordered set $\po$ equipped with a strictly monotone rank function with range $\{-1,0,\ldots,n\}$ (see McMullen \& Schulte~\cite[Chs. 2A,B]{arp}). The elements of ${\cal P}$ are called {\em faces\/}, or $j$-{\it faces} if their rank is $j$. Faces of ranks 0, 1 or $n-1$ are also called {\it vertices}, {\it edges} or {\it facets} of $\cal P$, respectively. Moreover, ${\cal P}$ has a smallest face (of rank $-1$) and largest face (of rank $n$), denoted by $F_{-1}$ and $F_n$, respectively; these are the {\it improper faces} of ${\cal P}$. Each {\it flag} (maximal totally ordered subset) $\Phi$ of ${\cal P}$ contains exactly $n+2$ faces, one of each rank $j$; this $j$-face is denoted by $(\Phi)_j$. Two flags are said to be {\it adjacent} if they differ in just one face; they are {\em $j$-adjacent\/} if this face has rank $j$. In $\cal P$, any two flags $\Phi$ and $\Psi$ can be joined by a sequence of flags $\Phi=\Phi_0,\Phi_1,...,\Phi_k=\Psi$, all containing $\Phi \cap \Psi$, such that any two successive flags $\Phi_{i-1}$ and $\Phi_i$ are adjacent; this property is known as the {\it strong flag-connectedness} of ${\cal P}$. Finally, $\cal P$ has the following {\it homogeneity  property}, usually referred to as the {\it diamond condition}:\  whenever $F\leq F'$, with ${\rm rank}(F)=j-1$ and ${\rm rank}(F')=j+1$, there are exactly two faces $F''$ of rank $j$ such that $F\leq F''\leq F'$. Then, for each flag $\Phi$ of $\po$ and each $j$ with $j=0, \dots, n-1$, there exists a unique $j$-adjacent flag of $\Phi$, denoted $\Phi^j$. 

If $F$ and $F'$ are faces of ${\mathcal P}$ with $F \leq F'$, then $F'/F:=\{F'' \in \po \mid F \leq F'' \leq F'\}$ is a polytope of rank ${\rm rank}(F')-{\rm rank}(F)-1$ called a {\em section\/} of $\po$. If $F$ is a vertex of $\po$, then $F_n/F$ is the {\em vertex-figure\/} of $\po$ at $F$. A {\em chain\/} (totally ordered subset) of $\po$ is said to be of {\em type\/} $J$, with $J \subseteq \{0, \dots, n-1\}$, if $J$ is the set of ranks of proper faces contained in it. Flags are maximal chains and are of type $\{0,\ldots,n-1\}$.

Following standard terminology for convex polytopes (see \cite{gru}), an abstract $n$-polytope is called {\em simple\/} if the vertex-figure at each vertex is isomorphic to an $(n-1)$-simplex. (Note here that the terms ``simple" in polytope theory and graph theory refer to different properties. An abstract polytope with an edge graph that is a ``simple graph" is usually not a simple polytope. However, a simple polytope has a ``simple" edge graph.)

A polytope $\po$ is {\em regular} if its automorphism group $\Gamma({\cal P})$ is transitive on the flags of ${\cal P}$. We say that ${\cal P}$ is {\em $j$-face transitive\/} if $\Gamma({\cal P})$ acts transitively on the $j$-faces of ${\cal P}$. Every regular polytope is $j$-face transitive for each $j$; however, a polytope that is $j$-face transitive for each $j$ need not be regular. Moreover, for $J \subseteq \{0, \dots, n-1\}$ we say that $\po$ is $J$-transitive if $\Gamma(\po)$ is transitive on the chains of $\po$ of type $J$. Note that, if $\po$ is $J$-transitive, then $\po$ is $j$-face transitive for each $j \in J$; however the converse is not true in general.

We now briefly review the construction of the {\em graphicahedron} ${\cal P}_G$, an abstract polytope associated with a finite graph $G$ (see \cite{graphi} for details). 

By a $(p,q)$-graph we mean a finite simple graph (without loops or multiple edges) that has $p$ vertices and $q$ edges.
Let $G$ be a $(p,q)$-graph with vertex set $V(G):=\{1,\ldots,p\}$ and edge set $E(G)=\{e_1,\ldots,e_q\}$, where $p\geq 1$ and $q\geq 0$ (if $q=0$, then $E(G)=\emptyset$). If $e = \{i,j\}$ is an edge of $G$ (with vertices $i$ and $j$), define the transposition $\tau_e$ in the symmetric group $S_p$ by ${\tau}_{e}:= (i\;j)$. Let $T_G:=\langle\tau_{e_1},\ldots,\tau_{e_q}\rangle$; this is the subgroup of $S_p$ generated by the transpositions determined by the edges of $G$. If $G$ is connected, then $T_{G}=S_p$. More generally, if $K\subseteq E(G)$ we let $T_K:=\langle\tau_{e}\mid e\in K\rangle$; this is the trivial group if $K=\emptyset$.

In defining $\mathcal{P}_G$ we initially concentrate on the faces of ranks $0,\ldots,q$ and simply append a face of rank $-1$ at the end. For $i\in I:=\{0,\ldots,q\}$ define 
\begin{equation}
\label{facets}
C_i :=  \{ (K,\alpha) \mid K\subseteq E(G), |K| = i,\, \alpha\in S_p \} .
\end{equation} 
We say that two elements $(K,\alpha)$ and $(L,\beta)$ of ${\bigcup _{i\in I} C_i}$ are {\em equivalent\/}, or $(K,\alpha)\sim(L,\beta)$ for short, if and only if $K=L$ and $T_{K}\alpha = T_{L}\beta$ (equality as right cosets in $S_p$). This defines an equivalence relation $\sim$ on ${\bigcup _{i\in I} C_i}$. We use the same symbol for both, an element of ${\bigcup _{i\in I} C_i}$ and its equivalence class under $\sim$. Then, as a set (of faces, of rank $\geq 0$), the graphicahedron $\mathcal{P}_G$ is just the set of equivalence classes
\[ {\big(\bigcup _{i\in I} C_i\big)}\slash{\sim} \] 
for $\sim$. The partial order $\leq$ on ${\cal P}_G$ is defined as follows: when $(K,\alpha), (L,\beta) \in {\cal P}_G$ then $(K,\alpha) \leq (L,\beta)$ if and only if $K\subseteq L$ and $T_{K}\alpha \subseteq T_{L}\beta$. This gives a partially ordered set with a rank function given by $\mathrm{rank}(K,\alpha) = |K|$. In \cite[Section 4]{graphi} it is shown that ${\cal P}_G$, with a $(-1)$-face appended, is an abstract q-polytope. 

Note that, if $(K,\alpha)$ and $(L,\beta)$ are two faces of ${\cal P}_G$ with $(K,\alpha) \leq (L,\beta)$, then $(L,\beta)=(L,\alpha)$ in ${\cal P}_G$, so in representing the larger face we may replace its second component $\beta$ by the second component $\alpha$ of the smaller face. This implies that a flag $\Phi$ of ${\cal P}_G$ can be described by two parameters, namely a maximal nested family of subsets of $E(G)$, denoted by ${\cal K}_{\Phi }:={\cal K}:=\{K_0,K_1,\ldots,K_q\}$, and a single element $\alpha\in S_p$ given by the second component of its vertex; that is, 
\[ \Phi = ({\cal K},\alpha) := \{(K_0,\alpha),(K_1,\alpha),\ldots,(K_q,\alpha)\} . \]
By a {\em maximal nested\/} family of subsets of a given finite set we mean a flag in the Boolean lattice (power set) associated with this set. (Each $K_i$ contains exactly $i$ edges, so $K_{0}=\emptyset$ and $K_{q}=E(G)$.) 

We next investigate face transitivity properties of the automorphism group $\Gamma(\po_G)$ of the graphicahedron. We know from \cite{graphi} that, for $q\neq 1$,
\[\Gamma(\po_G)=S_{p}\ltimes\Gamma(G)\]
the semidirect product of $S_p$ with the graph automorphism group $\Gamma(G)$ of $G$. In particular, if $\gamma \in S_p$, $\kappa \in \Gamma(G)$, and $(K, \alpha)$ is a face of $\po_G$, then the actions of $\gamma$ and $\kappa$ on $\po_G$ are given by 
\begin{eqnarray}
(K, \alpha) & \xrightarrow{\gamma}& (K, \alpha\gamma^{-1}), \label{s_p-action} \\
(K, \alpha) &\xrightarrow{\kappa} &  (\kappa(K), \alpha^\kappa),  \label{graph-action}
\end{eqnarray} 
with $\alpha^\kappa := \kappa\alpha\kappa^{-1}$. Note here that every graph automorphism $\kappa$ of $G$ is an incidence-preserving permutation of the $p$ vertices and $q$ edges of $G$ and hence determines two natural mappings:\  first, $\kappa$ acts on the subsets of $E(G)$ via its action on the edges of $G$; and second, $\kappa$ induces a group automorphism of $S_p$ via its action on the vertices of $G$, namely through conjugation in $S_p$ by $\kappa$. In (\ref{graph-action}), the actions of $\kappa$ on both the edges and the vertices of $G$ are employed. The subgroup $S_p$ of $\Gamma(\po_G)$ acts simply transitively on the vertices of $\po_G$. The stabilizer of the vertex $(\emptyset,\epsilon)$ of $\po_G$ in $\Gamma(\po_G)$ is the subgroup $\Gamma(G)$.

\begin{lemma}
\label{0,j-trans}
When $j\geq 1$ the graphicahedron $\po_G$ is $j$-face transitive if and only if $\po_G$ is $\{0,j\}$-transitive.
\end{lemma}

\begin{proof}
Clearly, a $\{0,j\}$-transitive polytope $\po_G$ is also $j$-face transitive. Now suppose $\po_G$ is $j$-face transitive and consider two $\{0,j\}$-chains $\{(\emptyset, \alpha),(K, \alpha)\}$ and $\{(\emptyset, \beta),(L, \beta)\}$ of $\po_G$, where $\alpha, \beta \in S_p$ and $|K| = |L| = j$. Then there exists an element $\rho\in \Gamma(\mathcal{P}_G)$ such that $\rho(K,\alpha)=(L,\beta)$. Writing $\rho= \gamma\kappa$, with $\gamma \in S_p$ and $\kappa \in \Gamma(G)$, we then obtain \[(L,\beta)=\rho(K,\alpha) = \gamma\kappa(K,\alpha)= (\kappa(K),\alpha^{\kappa}\gamma^{-1})\] 
and therefore $\kappa(K)=L$ and $T_{L}\beta = T_{L}\alpha^{\kappa}\gamma^{-1}$. It follows that $\alpha^{\kappa}\gamma^{-1}=\tau\beta$ for some $\tau\in T_L$. Then $\sigma:=\beta^{-1}\tau\beta$ lies in $S_p$ and can be viewed as an automorphism of $\po_G$. In particular,
\[ \sigma(\rho(K,\alpha)) = \sigma(L,\beta) = \sigma(L,\tau\beta)=(L,\tau\beta\sigma^{-1})=(L,\beta)\]
and
\[ \sigma(\rho(\emptyset,\alpha)) = \sigma(\emptyset,\alpha^{\kappa}\gamma^{-1})=\sigma(\emptyset,\tau\beta)
= (\emptyset,\tau\beta\sigma^{-1})=(\emptyset,\beta) .\] 
Hence the automorphism $\sigma\rho$ of $\po_G$ maps $\{(\emptyset, \alpha),(K, \alpha)\}$ to $\{(\emptyset, \beta),(L, \beta)\}$. Thus $\po_G$ is $\{0,j\}$-transitive.
~\end{proof}

A {\em spanning subgraph\/} of a graph $G$ is a subgraph containing all the vertices of $G$. It is completely characterized by its edge set. By a {\em $j$-subgraph\/} of $G$ we mean a spanning subgraph of $G$ with exactly $j$ edges.  If $H$ is a $j$-subgraph of $G$ with edge set $E(H)$ for some $j$, then the $(q-j)$-subgraph $H^c$ with edge set $E(G)\setminus E(H)$ is the {\em complement\/} of $H$ in $G$.

We call a graph $G$ {\em $j$-subgraph transitive\/} if its graph automorphism group $\Gamma(G)$ acts transitively on the $j$-subgraphs of $G$. Thus $G$ is $j$-subgraph transitive if and only if, for every $K, L \subseteq E(G)$ with $|K|=j=|L|$, there exists an automorphism $\kappa$ of $G$ such that $\kappa(K)=L$. Note that $G$ is $j$-subgraph transitive if and only if $G$ is $(q-j)$-subgraph transitive. A $1$-subgraph transitive graph is just an edge transitive graph. Every graph, trivially, is $j$-subgraph transitive for $j=0$ and $j=q$, since the trivial graph (with vertex set $V(G)$) and $G$ itself are the only $j$-subgraphs for $j=0$ or $q$, respectively.  A $j$-subgraph transitive graph need not be $k$-subgraph transitive for $k<j$. 

\begin{lemma} 
\label{1tran}
When $j\geq 1$ the graphicahedron $\mathcal{P}_G$ is $j$-face transitive if and only if $G$ is $j$-subgraph transitive.
\end{lemma}

\begin{proof}
Suppose $\po_G$ is $j$-face transitive. Let $K, L \subseteq E(G)$ with $|K| = |L| = j$. Then, by Lemma~\ref{0,j-trans}, since $\po_G$ is also $\{0,j\}$-transitive, there exists an automorphism $\rho$ of $\po_G$ that maps the chain $\{(\emptyset,\epsilon), (K,\epsilon)\}$ of type $\{0,j\}$ to $\{(\emptyset,\epsilon), (L,\epsilon)\}$. By our previous discussion, since $\rho$ fixes the vertex $(\emptyset, \epsilon)$, we must have $\rho \in \Gamma(G)$, $\rho=\kappa$ (say), and hence $\kappa(K)=L$. Thus $G$ is $j$-subgraph transitive.

For the converse, let $(K,\alpha)$, $(L,\beta)$ be two $j$-faces of $\po_G$, so in particular $|K| = |L| = j$. Since $G$ is $j$-subgraph transitive, there exists a graph automorphism $\kappa$ of $G$ with $\kappa(K) = L$. Let $\gamma \in S_p$ such that $\alpha^\kappa= \beta\gamma$, and let $\rho$ be the automorphism of $\po_G$ defined by $\rho:= \gamma \kappa$. Then, $\rho(K, \alpha) = (\kappa(K), \alpha^\kappa \gamma^{-1})=(L,\beta)$. Hence $\po_G$ is $j$-face transitive.
\end{proof}

Now bearing in mind that $G$ is $j$-subgraph transitive if and only if $G$ is $(q-j)$-subgraph transitive, we immediately have the following consequence.

\begin{lemma} 
\label{qmjtran}
When $j\geq 1$ the graphicahedron $\mathcal{P}_G$ is $j$-face transitive if and only if $\mathcal{P}_G$ is $(q-j)$-face transitive.
\end{lemma}

It is immediately clear from Lemma~\ref{1tran} that a graphicahedron $\mathcal{P}_G$ is edge-transitive if and only if the underlying graph $G$ is edge-transitive. There is a wide variety of interesting edge-transitive graphs, but except for the $q$-cycles $C_q$ we will not further investigate them here. However, as we will see in Theorem~\ref{starchar}, when $2\leq j\leq q-2$ (and hence $q\geq 4$) there is just one connected $j$-subgraph transitive graph with $q$ edges, namely the $q$-star $K_{1,q}$ (with $p=q+1$);  this is $k$-subgraph transitive for $k=0,\ldots,q$. Note that the two connected graphs with three edges, $K_{1,3}$ and $C_3$, both are $j$-transitive for $j=0,1,2$. 

The proof of Theorem~\ref{starchar} requires the following lemma. For a finite graph $H$ and a vertex $v$ of $H$, we let $d_H(v)$ denote the degree of $v$ in $H$.

\begin{lemma}
\label{starlemma}
Let $G$ be a finite connected graph with $q$ edges, let $2\leq j\leq q$, let $q\geq 4$ when $j=2$, and let $G$ be $j$-subgraph transitive. If $G$ has a $j$-subgraph (with a connected component) isomorphic to $K_{1,j}$, then $G$ is isomorphic to $K_{1,q}$.
\end{lemma}

\begin{proof}
Let $H$ be a $j$-subgraph of $G$ (with a connected component) isomorphic to $K_{1,j}$, let $v_0$ denote the vertex of $H$ of degree $j$, and let $e_0$ be an edge of $H$ with vertex $v_0$. Then each vertex of $H$ adjacent to $v_0$ has degree $1$ in $H$; and all other vertices of $H$ distinct from $v_0$ are isolated vertices of $H$, of degree $0$ in~$H$. By the $j$-subgraph transitivity, if $e$ is an edge of $G$ not belonging to $H$, then the new $j$-subgraph $H(e):=(H\! \setminus\! e_0) \cup e$ (obtained by replacing $e_0$ by $e$) must be equivalent to $H$ under $\Gamma(G)$ and hence must have the same set of vertex degrees as $H$. 

We need to show that $v_0$ is a vertex of every edge of $G$. This trivially holds for the edges of $H$. Now suppose $e$ is an edge of $G$ not belonging to $H$. We want to prove that $v_0$ is a vertex of $e$. The case $j\geq 3$ is easy; in fact, if $v_0$ is not a vertex of $e$, then $v_0$ has a degree, namely $j-1$, in $H(e)$ distinct from $0$, $1$ and $j$, contradicting the fact that $H$ and $H(e)$ have the same set of vertex degrees.. 
The case $j=2$ is slightly different. Let $e_1$ denote the (unique) edge of $H$ distinct from $e_0$. Now suppose $v_0$ is not a vertex of $e$. Then $e$ must join the two vertices of $H$ distinct from $v_0$; otherwise, at least one of the two $2$-subgraphs $H(e)$ and $(H\! \setminus\! e_1) \cup e$ is not connected and hence cannot be equivalent to $H$ under $\Gamma(G)$. But since $q\geq 4$ when $j=2$, there is an edge $e'$ of $G$ distinct from $e_0$, $e_1$ and $e$. In particular, $e'$ has no vertex in common with at least one edge from among $e_0,e_1,e$; together with this edge, $e'$ forms a $2$-subgraph which is not connected. Again this is impossible and shows that $v_0$ must be a vertex of $e$. Hence, in either case, $v_0$ is a vertex of every edge of $G$. Thus $G$ is isomorphic to $K_{1,q}$. 
\end{proof}

The following theorem is of interest in its own right, independent of the study of graphicahedra. 

\begin{theorem}
\label{starchar}
Let $G$ be a finite connected graph with $q$ edges, let $2\leq j\leq q-2$, and let $G$ be $j$-subgraph transitive. Then $G$ is isomorphic to the $q$-star $K_{1, q}$.
\end{theorem}

\begin{proof}
Our theorem follows directly from Lemma~\ref{starlemma} if we can show that $G$ has a $j$-subgraph (with a connected component) isomorphic to $K_{1,j}$; note here that $q\geq 4$, by our assumptions on $j$. Clearly, $G$ has a $j$-subgraph of this kind if and only if $G$ has a vertex of degree at least $j$. Hence, since $G$ is also $(q-j)$-transitive, it suffices to find a vertex in $G$ of degree greater than or equal to $j$ or $q-j$.

To this end, let $H$ be any connected $j$-subgraph of $G$; this $j$-subgraph exists since $G$ is connected. Let $v_0$ be a vertex of $H$ of minimum positive degree in $H$, let $v_1$ be a vertex adjacent to $v_0$ in $H$, and let $e_0:=\{v_0,v_1\}$. By the $j$-subgraph transitivity, if $e$ is any edge of $G$ not belonging to $H$, then $H(e)=(H\! \setminus\! e_0) \cup e$ is a $j$-subgraph equivalent to $H$ under $\Gamma(G)$ and in particular must have the same number of isolated vertices and the same set of vertex degrees as $H$. Moreover, if $u$ is any vertex of $G$, then its degrees in $H(e)$ and $H$ are related as follows:\ if $u\neq v_0,v_1$ then   
\[ d_{H(e)}(u) = \Bigg\{ 
\begin{array}{ll}
d_H(u) +1 &\mbox{if } u \mbox{ is a vertex of } e, \\
d_H(u)      &\mbox{if } u \mbox{ is not a vertex of }e;
\end{array} \]
and if $u\in\{v_0,v_1\}$ then  
\[ d_{H(e)}(u) = \Bigg\{ 
\begin{array}{ll}
d_H(u)      &\mbox{if } u \mbox{ is a vertex of } e, \\
d_H(u) -1  &\mbox{if } u \mbox{ is not a vertex of }e.
\end{array} \]
Now there are two possibilities for the minimum (positive) degree of $H$, that is, for $d_H(v_0)$.

First suppose $d_H(v_0)\geq 2$. Then $H$ cannot have any isolated vertices; otherwise, if $u$ is an isolated vertex of $H$ and $e'$ an edge of $G$ with vertex $u$, then $d_{H(e')}(u)=1$, so the $j$-subgraph $H(e')$ has a vertex of degree $1$ while $H$ does not have any such vertex. Now we can argue as follows. If $e$ is any of the $q-j$ edges of $G$ not belonging to $H$, then $v_0$ must be a vertex of $e$; otherwise, $d_{H(e)}(v_0)$ is one less than the minimum positive degree, $d_H(v_0)$, of $H(e)$, which is impossible. Therefore $d_G(v_0)\geq q-j$, so we have found a vertex of $G$ with a degree at least $q-j$, as desired. 

Now suppose $d_H(v_0)=1$. If $H$ happens to have no isolated vertices, then each of the $q-j$ edges $e$ of $G$ not belonging to $H$ must again have $v_0$ as a vertex; otherwise, the $j$-subgraph $H(e)$ does have an isolated vertex, namely $v_0$.
Thus $d_G(v_0)\geq q-j$, so again we found the desired vertex in this case. If $H$ does have an isolated vertex, we proceed as follows. Let $u$ be an isolated vertex of $H$, and let $e$ be an edge of $G$ joining $u$ to a vertex of positive degree in $H$. We claim that then each vertex in $H$ adjacent to the neighbor $v_1$ of $v_0$ must have degree $1$. In fact, if $v_1$ is joined by an edge $e_1$ to a neighbor $v$ with $d_H(v)>1$, then $v\neq v_0$ and $(H\! \setminus\! e_1) \cup e$ is a $j$-subgraph of $G$ with
one less isolated vertex than $H$, which is impossible; note here that the isolated vertices of this subgraph are just those of $H$, except $u$. Hence each vertex in $H$ adjacent to $v_1$ must have degree $1$ in $H$, and in particular cannot be a neighbor of any vertex other than $v_1$. But since $H$ is connected, the edges of $H$ with vertex $v_1$ then must comprise all $j$ edges of $H$. Thus $d_H(v_1)\geq j$, so now we found a vertex of degree at least $j$. This completes the proof.
\end{proof}

\section{The graphicahedra for $K_{1,q}$ and $C_q$}
\label{starcycle}

In this section we determine the structure of the graphicahedron for the $q$-star $K_{1,q}$ (with $q+1$ vertices and $q$ edges, all emanating from a ``central" vertex) and the $q$-cycle $C_q$, in each case with $q\geq 3$. These graphs have graph automorphism groups isomorphic to $S_q$ and $D_q$ (the dihedral group of order $2q$), so their graphicahedra (of rank $q$) have polytope automorphism groups $S_{q+1} \rtimes S_q$ and $S_{q} \rtimes D_q$, respectively. We also know from \cite[Cor. 5.1]{graphi} that $\po_{K_{1,q}}$ is a regular $q$-polytope.

It is convenient to use an alternative description of a graphicahedron $\mathcal{P}_G$ based on the line graph $L(G)$ of the underlying graph $G$ rather than on $G$ itself. The {\em line graph\/} $L(G)$ of a given graph $G$ is the graph whose vertices, referred to as {\em nodes\/}, are in one-to-one correspondence with the edges of $G$, and whose edges, referred to as {\em branches\/}, connect two nodes of $L(G)$ precisely when the corresponding edges of $G$ have a vertex of $G$ in common. In particular, $L(K_{1,q})=K_q$, the complete graph on $q$ vertices, and $L(C_q)=C_q$. 

The construction of the graphicahedron $\mathcal{P}_G$ for a given $(p,q)$-graph $G$ proceeded from a Cayley-graph for the symmetric group $S_p$ determined by the generating set ${\cal T}_G$ of $S_p$ consisting of the transpositions associated with the edges of $G$; more explicitly, each edge $e = \{i,j\}$ of $G$ gave rise to the transposition ${\tau}_{e}:= (i\;j)$ in ${\cal T}_G$. When we transition from $G$ to its line graph $L(G)$, these transpositions are naturally associated with the nodes of $L(G)$. Thus $L(G)$ can be viewed as a {\em diagram\/} for the symmetric group $S_p$, in which the nodes are labeled by transpositions in $S_p$, and in which the branches represent pairs of non-commuting transpositions (their product necessarily has period $3$). Hence, the transpositions associated with a pair of nodes of $L(G)$ commute if and only if these nodes are not joined by a branch of $L(G)$. 

\subsection{The $K_{1,q\,}$-graphicahedron}
\label{k1qgraphi}

Let $G=K_{1,q}$, where $q\geq 3$, and let $1,...,q+1$ denote the vertices of $G$ such that $q+1$ is the central vertex. Let $e_1,\ldots,e_q$ denote the edges of $G$ such that $e_{i}=\{i,q+1\}$ for $i=1,\ldots,q$. We begin by constructing a certain regular $q$-polytope $\po$ and then establish that this actually is the $K_{1,q\,}$-graphicahedron

Now the underlying diagram $L(G)$ is the complete graph $K_q$ on $q$ nodes. Label each node of $L(G)$ with the label, $i$, of the corresponding edge $e_i$ of $G$ that defines it. Thus the nodes receive the labels $1,\ldots,q$. Then  associate with each node $i$ the transposition in $S_{q+1}$ defined by $\tau_{i}:=(i\;q+1)$; this is just the transposition $\tau_{e_i}$ in ${\cal T}_G$, renamed. It is straightforward to check that $\tau_1,\ldots,\tau_q$ generate~$S_{q+1}$. 

In terms of the generating transpositions $\tau_1,\ldots,\tau_q$, the group $S_{q+1}$ is abstractly defined by the relations
\begin{equation}
\label{present}
\tau_{i}^{2} = (\tau_{i}\tau_{j})^{3} = (\tau_{i}\tau_{j}\tau_{k}\tau_{j})^{2} = \epsilon. 
\quad (1\leq i,j,k\leq q;\; i,j,k \mbox{ distinct}) 
\end{equation}
This follows from \cite[Theorem 9G4]{arp}, applied here with $\mathcal{N}=\{1,\ldots,q\}$, $s=2$, and $\mathcal{E}$ and $\mathcal{F}$, respectively, given by the set of all $2$-element or $3$-element subsets of $\mathcal{N}$, and with generators $\tau_1,\ldots,\tau_q$ for the corresponding group $\Gamma$. Since this presentation is completely symmetric in the generators, it is immediately clear that $S_{q+1}$ admits involutory group automorphisms $\omega_{1},\ldots,\omega_{q-1}$ acting on the generators $\tau_1,\ldots,\tau_q$ according to
\begin{equation}
\label{actionomega}
\omega_{j}(\tau_{j}) = \tau_{j+1}, \;\; \omega_{j}(\tau_{j+1}) = \tau_{j}, \;\;
\omega_{j}(\tau_{i}) = \tau_{i} \mbox{ for } i\neq j,j+1. 
\end{equation}
In fact, these $\omega_j$ are actually inner automorphisms and can be realized in $S_{q+1}$ by conjugation with the transposition $\widehat{\omega}_j := (j\; j+1)$, for $j=1,\ldots,q-1$. Note that $\omega_{1},\ldots,\omega_{q-1}$ generate a group of automorphisms of $S_{q+1}$ isomorphic to~$S_q$. 

\begin{figure}[h!]
\centering
{\begin{picture}(180,200)
\put(10,10){
\begin{picture}(160,190)
\put(0,30){\circle*{4}}
\put(120,0){\circle*{4}}
\put(75,180){\circle*{4}}
\put(150,90){\circle*{4}}
\put(0,30){\line(4,-1){120}}  
\put(0,30){\line(5,2){150}}
\put(0,30){\line(1,2){75}}
\put(75,180){\line(1,-4){45}}
\put(75,180){\line(5,-6){75}}  
\put(120,0){\line(1,3){30}}   
\put(-9,24){\sst{1}}
\put(120,-12){\sst{2}}
\put(154,87){\sst{3}}
\put(73,186){\sst{4}}
\put(76,31){$2$}
\put(118,65){$2$}
\put(53,65){$2$}
\put(70,90){$2$}
\put(50,0){$\omega_1$}
\put(140,35){$\omega_2$}
\put(118,138){$\omega_3$}
\end{picture}}
\end{picture}}
\caption{The group $S_5$ obtained when $q=4$.}
\label{grpsixssrel}
\end{figure}
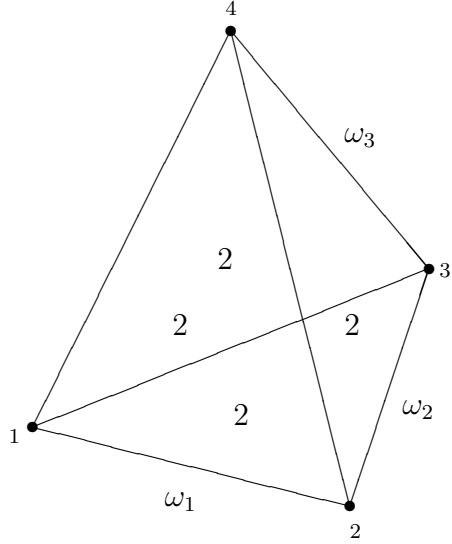

Figure~\ref{grpsixssrel} illustrates the situation when $q=4$. The nodes $1,\ldots,4$ of the diagram $K_4$ correspond to the generators $\tau_1,\ldots,\tau_4$ of $S_5$, respectively; a mark $2$ on a triangle $\{i,j,k\}$ of the diagram represents the corresponding non-Coxeter type relation (the last kind of relation) in (\ref{present}) involving the generators $\tau_i,\tau_j,\tau_k$. The automorphism $\omega_j$ then is associated with the branch connecting nodes $j$ and~$j+1$.

Thus we have an underlying group, $S_{q+1}$, whose distinguished generators are permuted under certain group automorphisms. This is the typical situation where the twisting technique of \cite[Ch. 8A]{arp} enables us to construct a regular polytope by extending the underlying group by certain of its group automorphisms. More specifically, in the present situation we find our regular $q$-polytope $\mathcal{P}$ by extending $S_{q+1}$ by $S_{q}$ as follows (see \cite[11H3]{arp} for the case $q=4$). Apply the ``twisting operation"
\begin{equation}
\label{twist}
(\tau_1,\ldots,\tau_q;\omega_1,\ldots,\omega_{q-1}) \mapsto 
(\tau_1,\omega_1,\ldots,\omega_{q-1}) 
=: (\sigma_0,\ldots,\sigma_{q-1})
\end{equation}
on $S_{q+1}=\langle\tau_1,\ldots,\tau_q\rangle$, in order to obtain the distinguished involutory generators $\sigma_0,\ldots,\sigma_{q-1}$ of the group $\Gamma(\mathcal{P})$ and hence the polytope $\mathcal{P}$ (of Schl\"afli type $\{6,3,\ldots,3\}$)  itself. Thus 
\[  \Gamma(\mathcal{P}) = \langle \sigma_{0},\ldots,\sigma_{q-1} \rangle
=  \langle \tau_1,\ldots,\tau_q \rangle \rtimes \langle \omega_1,\ldots,\omega_{q-1}\rangle
=  S_{q+1} \rtimes S_q,  \]
a semi-direct product. In particular, when $q=4$ we have $\Gamma(\mathcal{P})= S_{5} \rtimes S_4$ (see Figure~\ref{grpsixssrel}). Note here that the generators $\sigma_{0},\ldots,\sigma_{q-1}$ indeed make $S_{q+1} \rtimes S_q$ a string $C$-group in the sense of~\cite[Ch. 2E]{arp}; in particular, the proof of the intersection property boils down to observing that, for each $j\leq q-1$, 
\[ \langle\sigma_{0},\ldots,\sigma_{j-1}\rangle =
\langle \tau_1,\ldots,\tau_j \rangle \rtimes \langle \omega_1,\ldots,\omega_{j-1}\rangle 
= S_{j+1} \rtimes S_j, \]
and hence that  
\[ \begin{array}{rcl}
\langle\sigma_{0},\ldots,\sigma_{j-1}\rangle \cap \langle \sigma_{1},\ldots,\sigma_{q-1} \rangle \!\!
&=&\!\! (\langle \tau_1,\ldots,\tau_j \rangle \!\rtimes\! \langle \omega_1,\ldots,\omega_{j-1}\rangle) 
\cap \langle \omega_1,\ldots,\omega_{q-1}\rangle \\[.02in]
&=&\!\!\langle \omega_1,\ldots,\omega_{j-1}\rangle\\[.02in]
&=&\!\!\langle\sigma_{1},\ldots,\sigma_{j-1}\rangle .
\end{array} \]
 
Note that $\mathcal{P}$ is a simple polytope since its vertex-figures are isomorphic to $(q-1)$-simplices; in fact, the vertex-figure group is just $\langle \omega_1,\ldots,\omega_{q-1}\rangle = S_q$, with $\omega_1,\ldots,\omega_{q-1}$ occurring here as the distinguished generators for the group of the $(q-1)$-simplex. As for any regular polytope, the  vertices of $\mathcal{P}$ can be identified with the left cosets of its vertex-figure group, which here just means with the $(q+1)!$ elements of $S_{q+1}=\langle \tau_1,\ldots,\tau_q \rangle$. 

Finally, we need to prove that $\mathcal{P}$ is in fact the desired graphicahedron $\mathcal{P}_{G}$. Now, since both $\po$ and $\po_G$ are known to be regular, we can verify isomorphism on the level of regular polytopes where more techniques are available. In particular it suffices to establish an isomorphism between their groups that maps distinguished generators to distinguished generators.

To begin with, choose 
\[ \Phi := (\mathcal{L},\epsilon) =  \{(L_0,\epsilon),(L_1,\epsilon),\ldots,(L_q,\epsilon)\},\]
with $L_{i}:=\{e_1,\ldots,e_i\}$ for $i=0,\ldots,q$ and with $\mathcal{L}:=\{L_0, L_1, \dots, L_q\}$, as the base flag of $\po_G$ and compute the distinguished generators $\rho_0,\ldots,\rho_{q-1}$ of $\Gamma(\po_G)$ with respect to~$\Phi$. Since $\Gamma(\po_G)=S_{p}\rtimes \Gamma(G)$ (with $p=q+1$ and $\Gamma(G)\cong S_q$), we can write any isomorphism $\rho$ of $\po_G$ in the form $\rho=\gamma\kappa$ with $\gamma\in S_p$ and $\kappa\in \Gamma(G)$; then $\rho$ maps a face $(K,\alpha)$ of $\po_G$ to $(\kappa(K), \alpha^{\kappa}\gamma^{-1})$, with $\alpha^{\kappa}:=\kappa\alpha\kappa^{-1}$. Hence, if $\rho$ fixes each face of $\Phi$ except the $j$-face, then
$\kappa(L_i)=L_i$ and $\gamma\in T_{L_i}$ for each $i$ with $i\neq j$. Now, if $j=0$, then $\gamma=\tau_{e_1}$ (since $\gamma\in T_{L_1}=\langle\tau_{e_1}\rangle$) and $\kappa$ must be the trivial graph automorphism of $G$ (since it must fix every edge and $q\geq 2$). Thus $\rho_0=\tau_{e_1}$. Moreover, if $j\geq 1$, then $\gamma$ is the trivial element in $S_p$ (since $\gamma$ lies in the trivial group $T_{L_0}$) and $\kappa$ must be the graph automorphism $\kappa_j$ (say) of $G$ that interchanges the edges $e_{j}$ and $e_{j+1}$ while keeping all other edges of $G$ fixed. Thus $\rho_j=\kappa_j$ for $j=1,\ldots,q-1$.

Now we are almost done. Since $\Gamma(G)$ acts transitively on the edges of $G$ (while fixing the central vertex), the set of conjugates of $\tau_{e_1}=(1\;q+1)$ under $\Gamma(G)$ (now acting on the vertices of $G$) consists precisely of the transpositions $\tau_{e_j}=(j\;q+1)$ for $j=1,\ldots,q$. These transpositions generate the group $S_{p}$, with $p=q+1$. It is straightforward to check that they also satisfy the relations in (\ref{present}), which are known to form the relations in a presentation for $S_p$; as explained earlier, the latter follows from \cite[Theorem 9G4]{arp}. Thus the generators $\tau_{e_1},\ldots,\tau_{e_q}$ and the relations in (\ref{present}) determine a presentation for $S_p$, in just the same way as $\tau_1,\ldots,\tau_q$ and (\ref{present}) do. Moreover, conjugation in $S_p$ by the graph automorphism $\kappa_j$ (acting on vertices) of $G$ is a group automorphism of $S_p$, again denoted by $\kappa_j$, that permutes the generators $\tau_{e_1},\ldots,\tau_{e_q}$ in just the same way as $\omega_j$ permuted $\tau_{1},\ldots,\tau_{q}$; that is, 
\[ \kappa_{j}(\tau_{e_j}) = \tau_{e_{j+1}}, \;\; \kappa_{j}(\tau_{e_{j+1}}) = \tau_{e_j}, \;\;
\kappa_{j}(\tau_{e_i}) = \tau_{e_i} \mbox{ for } i\neq j,j+1. \]
It follows that the distinguished generators $\rho_0,\ldots,\rho_{q-1}$ of $\Gamma(\po_G)$ can be identified with those of $\Gamma(\po)$, in the sense that $\rho_i$ corresponds to $\sigma_i$ of (\ref{twist}) for each $i$. Since the distinguished generators of the group of a regular polytope determine the polytope up to isomorphism, it now follows that $\po_G$ is isomorphic to $\po$.

It is interesting to look at small values of $q$. The case $q=2$ had been excluded in the above; the $2$-star $K_{1,2}$ is a path of length $2$ and its graphicahedron $\po_{K_{1,2\,}}$ is just a hexagon $\{6\}$. When $q=3$ the resulting graphicahedron $\po_{K_{1,3\,}}$ is the toroidal regular map $\{6,3\}_{(2,2)}$ (see Coxeter \& Moser~\cite{cm}). For $q=4$ we obtain the universal locally toroidal regular $4$-polytope $\{\{6,3\}_{(2,2)},\{3,3\}\}$ (see~\cite[Cor. 11C8]{arp}), which was first discovered in Gr\"unbaum~\cite{grgcd}. 

The graphicahedra $K_{1,3\,}$ and $K_{1,4\,}$ are the first two members in an infinite sequence of polytopes linked inductively by amalgamation, as described in Theorem~\ref{staruniv}. We first require a lemma.

\begin{lemma}
\label{confsymgr}
Suppose $q\geq 4$ and $\mathcal{R}$ is the $(q-1)$-simplex with group 
$\Gamma(\mathcal{R})=\langle\eta_{1},\ldots,\eta_{q-1}\rangle\cong S_q$, where $\eta_{1},\ldots,\eta_{q-1}$ are the distinguished generators of $\Gamma(\mathcal{R})$, with the labeling starting at $1$ deliberately. Define the elements 
$\theta_{i}:= \eta_{i}\eta_{i-1}\ldots\eta_{1}$ for $i=1,\ldots,q-1$, as well as $\theta_{0}:=\epsilon$, the identity element. Then
\[ \eta_{j} \theta_{i} \,=\, \theta_{i}\eta_{j+1}, \theta_{i-1}, \theta_{i+1}, \theta_{i}\eta_j\]
according as $j\leq i-1$; $j=i$; $j=i+1$; $j\geq i+2$.
\end{lemma} 

\begin{proof}
One way to prove this lemma is to represent the generators $\eta_{j}$ by the transposition $(j\;j+1)$ of $S_q$ for each $j$ and then simply verify the statement on the level of permutations.

Alternatively we can use the defining relations of $\Gamma(\mathcal{R})$ as a Coxeter group (of type $A_{q-1}$). When  $j=i,i+1$ the statement is immediately clear from the definitions, and when $j\geq i+2$ it follows directly from the standard commutation relations for the generators of $\Gamma(\mathcal{R})$. Now when $j=i-1$ the relation $\eta_{i-1}\eta_{i}\eta_{i-1}= \eta_{i}\eta_{i-1}\eta_{i}$, combined with the commutation relations, gives
\[ \eta_{j} \theta_{i}=\eta_{i-1}\eta_{i}\eta_{i-1}\theta_{i-2}
=\eta_{i}\eta_{i-1}\eta_{i}\theta_{i-2}=\eta_{i}\eta_{i-1}\theta_{i-2}\eta_{i}
=\theta_{i}\eta_i . \]
Finally, when $j\leq i-2$ we similarly have
\[\begin{array}{rcl}
\eta_{j} \theta_{i} &=&\eta_{j}\eta_{i}\ldots\eta_{j+2}\eta_{j+1}\eta_{j}\eta_{j-1}\eta_{j-2}\ldots\eta_{1}\\[.02in]
&=& \eta_{i}\ldots\eta_{j+2}(\eta_{j}\eta_{j+1}\eta_{j})\eta_{j-1}\eta_{j-2}\ldots\eta_{1}\\[.02in]
&=& \eta_{i}\ldots\eta_{j+2}(\eta_{j+1}\eta_{j}\eta_{j+1})\eta_{j-1}\eta_{j-2}\ldots\eta_{1}\\[.02in]
&=& \eta_{i}\ldots\eta_{j+2}\eta_{j+1}\eta_{j}\eta_{j-1}\eta_{j-2}\ldots\eta_{1}\eta_{j+1}\\[.02in]
&=& \theta_{i}\eta_{j+1}.
\end{array} \] 
This completes the proof.
\end{proof}

Recall that if two regular $n$-polytopes $\mathcal{P}_1$ and $\mathcal{P}_2$ are the facets and vertex-figures, respectively, of a regular $(n+1)$-polytope, then there is also a universal such $(n+1)$-polytope, denoted $\{\mathcal{P}_1,\mathcal{P}_2\}$, of which any other regular $(n+1)$-polytope with facets $\mathcal{P}_1$ and vertex-figures $\mathcal{P}_2$ is a quotient (see \cite[p.97]{arp}). The locally toroidal $4$-polytope mentioned earlier is an example of a universal regular polytope, with facets $\{6,3\}_{(2,2)}$ and vertex-figures $\{3,3\}$.

\begin{theorem}
\label{staruniv}
The graphicahedron $\po_{K_{1,q\,}}$, with $q\geq 4$, is the universal regular $q$-polytope with graphicahedra $\po_{K_{1,q-1\,}}$ as facets and with $(q-1)$-simplices $\{3^{q-2}\}$ as vertex-figures. Thus $\po_{K_{1,q\,}} = \{\po_{K_{1,q-1\,}},\{3^{q-2}\}\}$. Moreover, $\Gamma(\po_{K_{1,q\,}}) = S_{q+1} \times S_{q}$, a direct product. 
\end{theorem}

\begin{proof}
Any graphicahedron is a simple polytope, and hence $\po_{K_{1,q\,}}$ has $(q-1)$-simplices $\{3^{q-2}\}$ as vertex-figures. (Recall from Section~\ref{dron} that an abstract polytope is simple precisely when each vertex-figure is isomorphic to a simplex.) The subgraphs of $K_{1,q}$ with $q-1$ edges are obtained from $K_{1,q}$ by the deletion of a single edge and in particular are subgraphs isomorphic to $K_{1,q-1}$. On the other hand, the facets of $\po_{K_{1,q\,}}$ are precisely determined by such subgraphs and hence are $K_{1,q-1\,}$-graphicahedra. Thus $\po_{K_{1,q\,}}$ is a regular polytope with graphicahedra $\po_{K_{1,q-1\,}}$ as facets and $(q-1)$-simplices $\{3^{q-2}\}$ as vertex-figures. In particular, $\po_{K_{1,q\,}}$ is a quotient of the universal polytope $\mathcal{Q}:=\{\po_{K_{1,q-1\,}},\{3^{q-2}\}\}$. We need to establish that $\po_{K_{1,q\,}}$ and $\mathcal{Q}$ are isomorphic when $q\geq 4$. (Note that the universality fails when $q=3$. Still, $\po_{K_{1,3\,}}=\{6,3\}_{(2,2)}$ is a $3$-polytope with facets $\po_{K_{1,2\,}}= \{6\}$ and vertex-figures $\{3\}$, and group $\Gamma(\po_{K_{1,3\,}}) = S_{4} \times S_{3}$.)

To this end, let $\Gamma:=\Gamma(\mathcal{Q})=\langle \eta_0,\ldots,\eta_{q-1}\rangle$, where $\eta_0,\ldots,\eta_{q-1}$ are the distinguished generators of $\Gamma(\mathcal{Q})$. Then $\Gamma_{0}:=\langle \eta_1,\ldots,\eta_{q-1}\rangle$ is the group of the vertex-figure $\{3^{q-2}\}$ of $\mathcal{Q}$ and is isomorphic to $S_{q}$. Now let $N_0$ denote the normal closure of $\eta_0$ in $\Gamma$, that is,  
\[ N_{0}:=\langle \phi\eta_0\phi^{-1} \mid \phi \in \Gamma\rangle . \] 
Then, by \cite[Lemma~4E7(c,d)]{arp} (applied with $k=0$ and $N_{0}^{-}=N_0$),
\[ N_{0}=\langle \phi\eta_0\phi^{-1} \mid \phi \in \Gamma_0\rangle  \] 
and $\Gamma = N_{0}\Gamma_{0} = \Gamma_{0}N_{0}$ (as a product of subgroups in $\Gamma$). Our goals are to show that $N_{0}\cong S_{q+1}$ and the product of subgroups is semi-direct.
 
Consider the following involutory elements $\mu_1,\ldots,\mu_{q}$ of $N_0$ defined by
\[ \mu_{i+1}:=\theta_{i}\eta_{0}\theta_{i}^{-1} \;\;(i=0,\ldots,q-1), \]
where $\theta_{0}:=\epsilon$ and $\theta_{i}:= \eta_{i}\eta_{i-1}\ldots\eta_{1}$ for $i=1,\ldots,q-1$. Then $\mu_{1}=\eta_0$ and $\mu_{i+1}=\eta_{i}\mu_{i}\eta_{i}$ for $i\geq 1$. We claim that $N_0$ is generated by $\mu_1,\ldots,\mu_{q}$. Clearly it suffices to show that the vertex-figure subgroup $\Gamma_0$ acts, by conjugation, as a permutation group on the set $\{\mu_1,\ldots,\mu_d\}$; more precisely, the conjugate of an element $\mu_{i+1}$ by a generator $\eta_j$ of $\Gamma_0$ is again a generator $\mu_k$ for some $k$. Now, by Lemma~\ref{confsymgr}, 
\[ \eta_{j}\mu_{i+1}\eta_{j}
= (\eta_{j}\theta_{i})\eta_{0}(\eta_{j}\theta_{i})^{-1} \;\;(i=0,\ldots,q-1; j=1,\ldots,q-1),\] 
with $\eta_{j}\theta_{i} = \theta_{i}\eta_{j+1}, \theta_{i-1}, \theta_{i+1}, \theta_{i}\eta_j$ according as 
$j\leq i-1$; $j=i$; $j=i+1$; $j\geq i+2$. Bearing in mind that $\eta_l$ and $\eta_0$ commute if $l\geq 2$, then this immediately implies  
\begin{equation}
\label{eqet}
\eta_{j}\mu_{i+1}\eta_{j} = \mu_{i},\mu_{i+2},\mu_{i+1}
\end{equation} 
according as $j=i$, $j=i+1$, or $j\neq i,i+1$. Thus 
\[ N_0=\langle\mu_1,\ldots,\mu_{q}\rangle .\]

Next we establish that $N_{0}$ is isomorphic to the symmetric group $S_{q+1}=\langle\tau_{1},\ldots,\tau_{q}\rangle$, where as before the generators are given by $\tau_{i}:=(i\;q+1)$ for $i=1,\ldots,q$. In particular we show that the map that takes $\mu_i$ to $\tau_i$ for each $i$ defines an isomorphism. As $S_{q+1}$ is abstractly defined by the presentation in (\ref{present}), it suffices to prove that the involutory generators $\mu_1,\ldots,\mu_{q}$ of $N_0$ also satisfy the relations in (\ref{present}). We already know at this point that, since $\mathcal{Q}$ covers $\po_{K_{1,q\,}}$, the particular subgroup $N_0$ of $\Gamma(\mathcal{Q})$ must cover the analogously defined subgroup of $\Gamma(\po_{K_{1,q\,}})$; however, 
this latter subgroup is just $S_{q+1}=\langle\tau_{1},\ldots,\tau_{q}\rangle$.

The key is to exploit the action of $\Gamma_0$ on $N_0$ in conjunction with the structure of the facet of $\mathcal{Q}$ as the $K_{1,q-1}$-graphicahedron. The latter allows us to proceed inductively. In fact, the subgroup $\langle\mu_1,\ldots,\mu_{q-1}\rangle$ of $N_0$ is just the subgroup, analogous to $N_0$, of the facet subgroup $\langle\eta_{0},\ldots,\eta_{q-2}\rangle$ of $\Gamma$. (This is also true when $q=4$, since then $\po_{K_{1,3}} = \{6,3\}_{(2,2)}$ and $\Gamma(\po_{K_{1,3}})= S_{4}\rtimes S_{3}$.) Thus, by induction hypothesis, we can take for granted all relations of  (\ref{present}) that involve only the generators $\mu_1,\ldots,\mu_{q-1}$. Now if a relator of (\ref{present}) involves the last generator $\mu_q$ of $N_0$, we first find a conjugate relator by an element of $\Gamma_0$ that does not contain $\mu_q$, and then simply refer to the validity of the corresponding relation for this conjugate.

The Coxeter type relations $(\mu_{i}\mu_{q})^{3}=\epsilon$ for $i=1,\ldots,q-1$ can be verified as follows. We indicate conjugacy of elements by $\sim$. If $i\leq q-2$ then 
\[ \mu_{i}\mu_{q} \sim 
\eta_{q-1}\mu_{i}\eta_{q-1} \!\cdot\! \eta_{q-1}\mu_{q}\eta_{q-1}
=\mu_{i}\mu_{q-1}, \]
where now the element on the right is known to have order $3$. Similarly, by (\ref{eqet}),
\[ \mu_{q-1}\mu_{q} 
\sim\eta_{q-1}\eta_{q-2}\mu_{q-1}\eta_{q-2}\eta_{q-1} \!\cdot\! \eta_{q-1}\eta_{q-2}\mu_{q}\eta_{q-2}\eta_{q-1} 
\sim\mu_{q-2}\mu_{q-1}, \]
where again the element on the right is known to have order $3$. 

Since the pairwise products of the generators of $N_0$ all have order $3$, the six relators $\mu_{i}\mu_{j}\mu_{k}\mu_{j}$ of the remaining relations that only involve the suffices $i,j,k$ (in some order) turn out to be mutually conjugate, for any distinct $i,j,k$ (see \cite[p. 295]{arp}). Hence we may assume that $i<k<j=q$. Now when $k\leq q-2$ we have 
\[ \mu_{i}\mu_{q}\mu_{k}\mu_{q} 
\sim \eta_{q-1}\,\mu_{i}\mu_{q}\mu_{k}\mu_{q}\,\eta_{q-1} 
= \mu_{i} \mu_{q-1}\mu_{k}\mu_{q-1}; \]
and when $i\leq q-3$ and $k=q-1$ then 
\[ \mu_{i}\mu_{q}\mu_{q-1}\mu_{q} 
\sim \eta_{q-1}\eta_{q-2}\;\mu_{i}\mu_{q}\mu_{q-1}\mu_{q}\;\eta_{q-2}\eta_{q-1}\
= \mu_{i}\mu_{q-1}\mu_{q-2}\mu_{q-1}. \]
Finally, when $i=q-2$ and $k=q-1$ we similarly obtain
\[ \mu_{q-2}\mu_{q}\mu_{q-1}\mu_{q} 
\sim \eta_{q-1}\eta_{q-2}\eta_{q-3}\;\mu_{q-2}\mu_{q}\mu_{q-1}\mu_{q}\;\eta_{q-3}\eta_{q-2}\eta_{q-1}\
= \mu_{q-3}\mu_{q-1}\mu_{q-2}\mu_{q-1}; \]
the proof of this last equation uses the fact that $\eta_{q-1}\eta_{q-2}\eta_{q-3}$ has period $4$ (since the $3$-simplex $\{3,3\}$ has Petrie polygons of length $4$), so that 
\[ \begin{array}{rcl}
\eta_{q-1}\eta_{q-2}\eta_{q-3}\; \mu_{q}\; \eta_{q-3}\eta_{q-2}\eta_{q-1} 
&=& (\eta_{q-1}\eta_{q-2}\eta_{q-3})^{2}\; \mu_{q-3}\; (\eta_{q-3}\eta_{q-2}\eta_{q-1})^{2} \\
&=& (\eta_{q-3}\eta_{q-2}\eta_{q-1})^{2}\; \mu_{q-3}\; (\eta_{q-1}\eta_{q-2}\eta_{q-3})^{2} \\
&=& \eta_{q-3}\eta_{q-2}\eta_{q-1}\eta_{q-3}\; \mu_{q-3}\; \eta_{q-3}\eta_{q-1}\eta_{q-2}\eta_{q-3}\\
&=& \eta_{q-3}\eta_{q-2}\eta_{q-3}\eta_{q-1}\; \mu_{q-3}\; \eta_{q-1}\eta_{q-3}\eta_{q-2}\eta_{q-3}\\
&=& \eta_{q-2}\eta_{q-3}\eta_{q-2}\eta_{q-1}\; \mu_{q-3}\; \eta_{q-1}\eta_{q-2}\eta_{q-3}\eta_{q-2}\\
&=& \eta_{q-2}\eta_{q-3}\; \mu_{q-3}\; \eta_{q-3}\eta_{q-2}\\
&=& \mu_{q-1} .
\end{array} \]
In either case, the conjugate of the original relator $\mu_{i}\mu_{q}\mu_{k}\mu_{q}$ is a relator of the same kind that only involve suffices from among $1,\ldots,q-1$ and hence must have period $2$, as desired. Thus $N_0$ is isomorphic to $S_{q+1}$ under the isomorphism that takes $\mu_i$ to $\tau_i$ for each $i$.

We now apply a simple counting argument to prove that $\mathcal{Q}$ and $\mathcal{P}_{K_{1,q}}$ are isomorphic polytopes. In fact, since $\Gamma = N_{0}\Gamma_{0}$, $N_{0}\cong S_{q+1}$ and $\Gamma_{0}\cong S_{q}$, the order of $\Gamma$ is at most $(q+1)!q!$. On the other hand, the group of $\mathcal{P}_{K_{1,q}}$ has order $(q+1)!q!$ and is a quotient of $\Gamma$. Hence, $\Gamma\cong \Gamma(\mathcal{P}_{K_{1,q}})$, $\Gamma\cong N_{0} \rtimes \Gamma_0$, and $\mathcal{Q}$ is isomorphic to $\mathcal{P}_{K_{1,q}}$. 

Finally, $\Gamma$ is also isomorphic to a direct product $S_{q+1}\times S_q$ (with a new second factor). The proof  follows the same line of argument as the proof for $q=4$ described in \cite[p. 405]{arp}. To this end, we embed the semi-direct product $S_{q+1}\rtimes S_q$ into the symmetric group $S_{2q+1}$ by taking $\tau_{i}=(i\;q+1)$ for $i=1,\dots,q$ and $\omega_{i}=(i\;\,i+1)(q+i+1\;\,q+i+2)$ for $i=1,\dots,q-1$, the latter viewed as acting on $\langle\tau_1,\ldots,\tau_q\rangle$ by conjugation; thus this subgroup of $S_{2q+1}$ is isomorphic to $S_{q+1}\times S_q$. This completes the proof.
\end{proof}

\subsection{The $C_q$-graphicahedron}
\label{cqgraphi}

The main goal of this section is to prove Theorem~\ref{voronoi} below, which says in particular that the $C_q$-graphicahedron can be viewed as a tessellation by $(q-1)$-dimensional permutahedra on the $(q-1)$-torus.

For the $q$-cycle $G=C_q$ we again employ the diagram $L(G)$ rather than $G$ itself. Clearly, as graphs, $G$ and $L(G)$ are isomorphic in this case, but in $L(G)$ the nodes, rather than the edges, are associated with transpositions. As before, the vertices of $G$ are labeled $1,\ldots,q$ (now $p=q$), in cyclic order; the edges of $G$ are labeled $e_1,\ldots,e_q$, with $e_{i}=\{i,i+1\}$ for $i=1,\ldots,q$ (labels are considered modulo $q$);  and each node of $L(G)$ receives the label, $i$, of the edge $e_i$ of $G$ that defines it. Thus the nodes are also labeled $1,\ldots,q$. 

Now the transposition associated with the node $i$ is given by $\tau_{i}:=(i\;i+1)$; as before, this is just $\tau_{e_i}$ in $\mathcal{T}_G$, renamed. Then it is straightforward to verify that $\tau_1,\ldots,\tau_q$ generate $S_{q}$ and also satisfy the Coxeter relations
\begin{equation}
\label{cox}
\tau_{i}^{2} = (\tau_{i}\tau_{i+1})^{3} = (\tau_{i}\tau_{j})^{2} = \epsilon 
\quad (1\leq i,j\leq q;\; j\neq i, i\pm 1;\, \mbox{subscripts}\bmod q), 
\end{equation}
as well as the extra relation
\begin{equation}
\label{extra}
\tau_{1}\tau_{2}\ldots \tau_{q-1}\tau_{q}\tau_{q-1}\ldots \tau_{2}=\epsilon
\end{equation}
(which expresses $\tau_1$ or $\tau_q$ in terms of the other generators and shows redundancy in the generating set). We explain later why the combined relations in (\ref{cox}) and (\ref{extra}) form a complete presentation for the group $S_q$, although we will not actually require this fact.

The key idea of the proof of Theorem~\ref{voronoi} is to exploit the geometry of the infinite Euclidean Coxeter group $\tilde{A}_{q-1}$, whose Coxeter diagram, usually also denoted $\tilde{A}_{q-1}$, is the $q$-cycle $C_q$ (see \cite{cs,RP} and \cite[Ch. 3B]{arp}). Thus $\tilde{A}_{q-1}$ is the group with generators $r_1,\ldots,r_q$ (say) abstractly defined by the Coxeter relations in equation (\ref{cox}); as we will see later, its element 
\[ r_{1}r_{2}\ldots r_{q-1}r_{q}r_{q-1}\ldots r_{2} \]
corresponding to the element on the left side of (\ref{extra}) has infinite order and represents a translation. Each (maximal parabolic) subgroup of $\tilde{A}_{q-1}$ generated by all but one generator $r_i$ is a finite Coxeter group of type $A_{q-1}$ and is isomorphic to $S_q$. Clearly, since the generators $\tau_1,\ldots,\tau_q$ of $S_q$ also satisfy the relations in (\ref{cox}), there exists an epimorphism 
\begin{equation}
\label{maptau}
\pi:  \tilde{A}_{q-1}\rightarrow S_q
\end{equation} 
that maps $r_j$ to $\tau_j$ for each $j$. Moreover, for each $i$, the restriction of $\pi$ to the Coxeter subgroup $\langle r_j\mid j\neq i\rangle$ is an isomorphism onto $S_q$.

As a geometric group, $\tilde{A}_{q-1}$ acts on Euclidean $(q-1)$-space $\mathbb{E}^{q-1}$ as an infinite discrete group of isometries generated by $q$ hyperplane reflections, for simplicity again denoted $r_1,\ldots,r_q$.  The $q$ hyperplane mirrors are the ``walls" bounding a fundamental $(q-1)$-simplex (fundamental chamber) $S$ for $\tilde{A}_{q-1}$ in its action on $\mathbb{E}^{q-1}$. The vertex of $S$ opposite to the reflection mirror of $r_i$ naturally receives the label $i$ in this way. The images of $S$ under $\tilde{A}_{q-1}$ are the $(q-1)$-simplices (chambers) of the Coxeter complex $\mathcal{C}:=\mathcal{C}(\tilde{A}_{q-1})$ of $\tilde{A}_{q-1}$, which is a $(q-1)$-dimensional simplicial complex tiling the entire space $\mathbb{E}^{q-1}$. The Coxeter group $\tilde{A}_{q-1}$ acts simply transitively on the chambers of $\mathcal{C}$; so in particular, each chamber of $\mathcal{C}$ can be identified with an element of $\tilde{A}_{q-1}$ and then labeled, via $\pi$, with a permutation from $S_q$.  Moreover, $\mathcal{C}$ has the structure of a {\em (vertex-) labeled simplicial complex\/}, with a vertex-labeling which is induced by the action of $\tilde{A}_{q-1}$ on $\mathcal{C}$ and is compatible with the vertex labeling on the fundamental simplex $S$ described above; that is, if an element of $\tilde{A}_{q-1}$ maps $S$ to a chamber $S'$, then for each $i$ the vertex of $S$ labeled $i$ is mapped to the vertex of $S'$ labeled~$i$. (Note that the vertex-labeling on $\mathcal{C}$ also induces a labeling of the entire Coxeter complex where a face of $\mathcal{C}$ receives, as its label, the subset of $\{1,\ldots,q\}$ consisting of the labels of its vertices. However, we will not require this labeling.)

For example, for the $3$-cycle $C_3$ (that is, when $q=3$), the corresponding Coxeter group $\tilde{A}_{2}=\langle r_1,r_2,r_3\rangle$ is the Euclidean reflection group in $\mathbb{E}^2$ generated by the reflections in the walls (edges) of an equilateral triangle, $S$, such that $r_i$ is the reflection in the wall of $S$ opposite to the vertex labeled $i$ (see Figures~\ref{A2tilde1} and~\ref{A2tilde2}). Now the corresponding Coxeter complex $\mathcal{C}$ is just the standard regular tessellation of $\mathbb{E}^2$ by equilateral triangles (chambers), and any such triangle can serve as fundamental chamber. Each triangle of $\mathcal{C}$ is marked with a permutation in $S_3$ derived from the epimorphism $\pi: \tilde{A}_{2}\rightarrow S_3$. The mark on a triangle $S'$ of $\mathcal{C}$ is the image under $\pi$ of the unique element of $\tilde{A}_2$ that maps $S$ to $S'$, as illustrated in Figure~\ref{A2tilde2}. In particular, the mark on $S$ is the identity permutation, denoted~$(1)$. Note that the chambers in $\mathcal{C}$ are not regular simplices when $q>3$.
\medskip

\begin{figure}[ht!]
\begin{center}
\scalebox{0.38}{\includegraphics{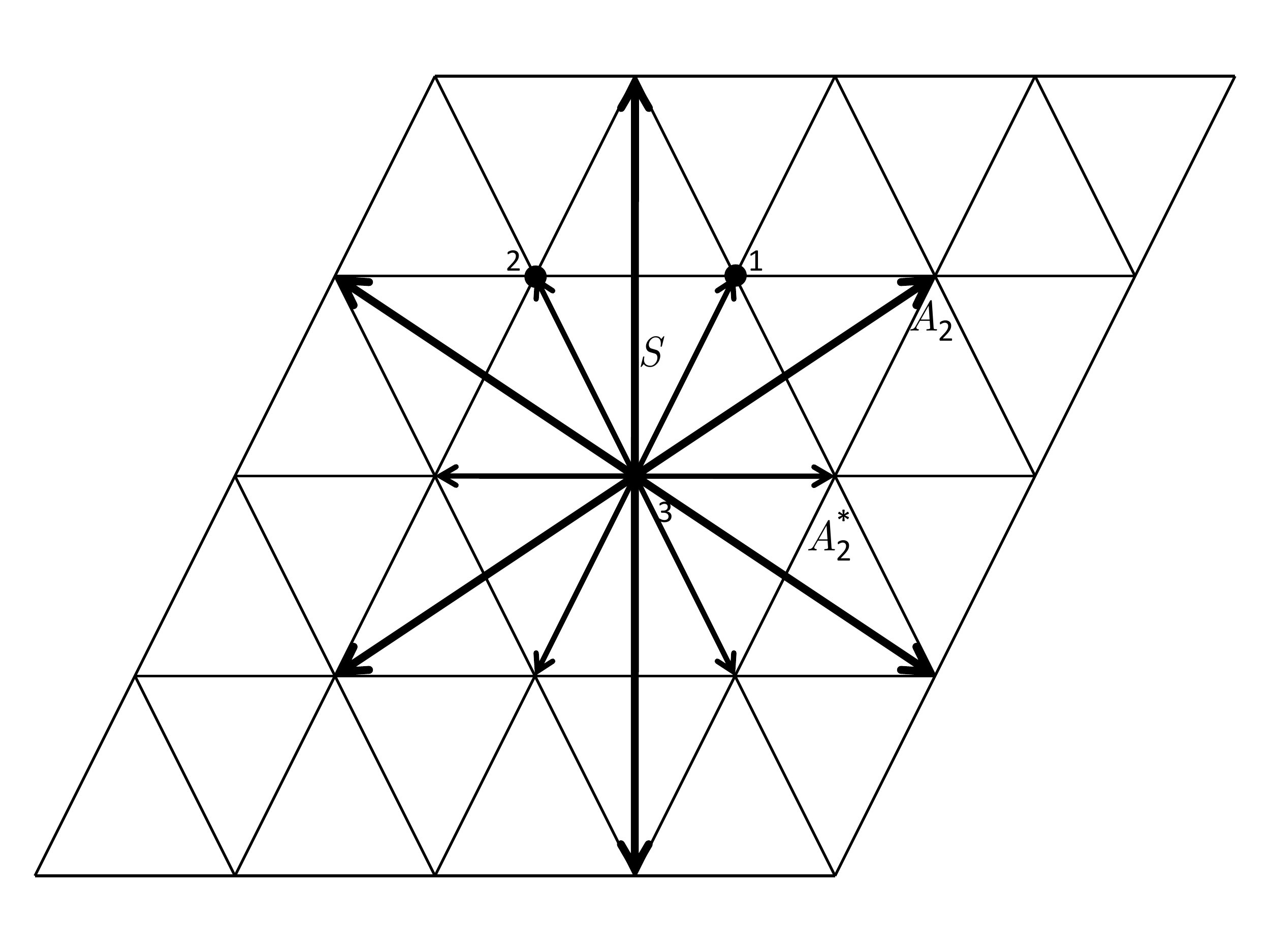}} \\[-.2in]
\caption{\it The case $q=3$, illustrated in $\mathbb{E}^2$. Shown are the Coxeter complex (triangle tessellation) $\mathcal{C}$ for the Coxeter group $\tilde{A}_2=\langle r_{1},r_{2},r_{3}\rangle$, with fundamental triangle $S$, as well as the shortest vectors of the root lattice $A_2$ and its dual lattice $A_{2}^*$. The vertex of $S$ labeled $3$ is taken to be the origin. The generator $r_i$ of $\tilde{A}_2$ is the reflection in the line through the edge (wall) of $S$ opposite the vertex of $S$ labeled $i$, for $i=1,2,3$.} \label{A2tilde1}
\end{center}
\end{figure}

\begin{figure}[ht!]
\begin{center}
\scalebox{0.54}{\includegraphics{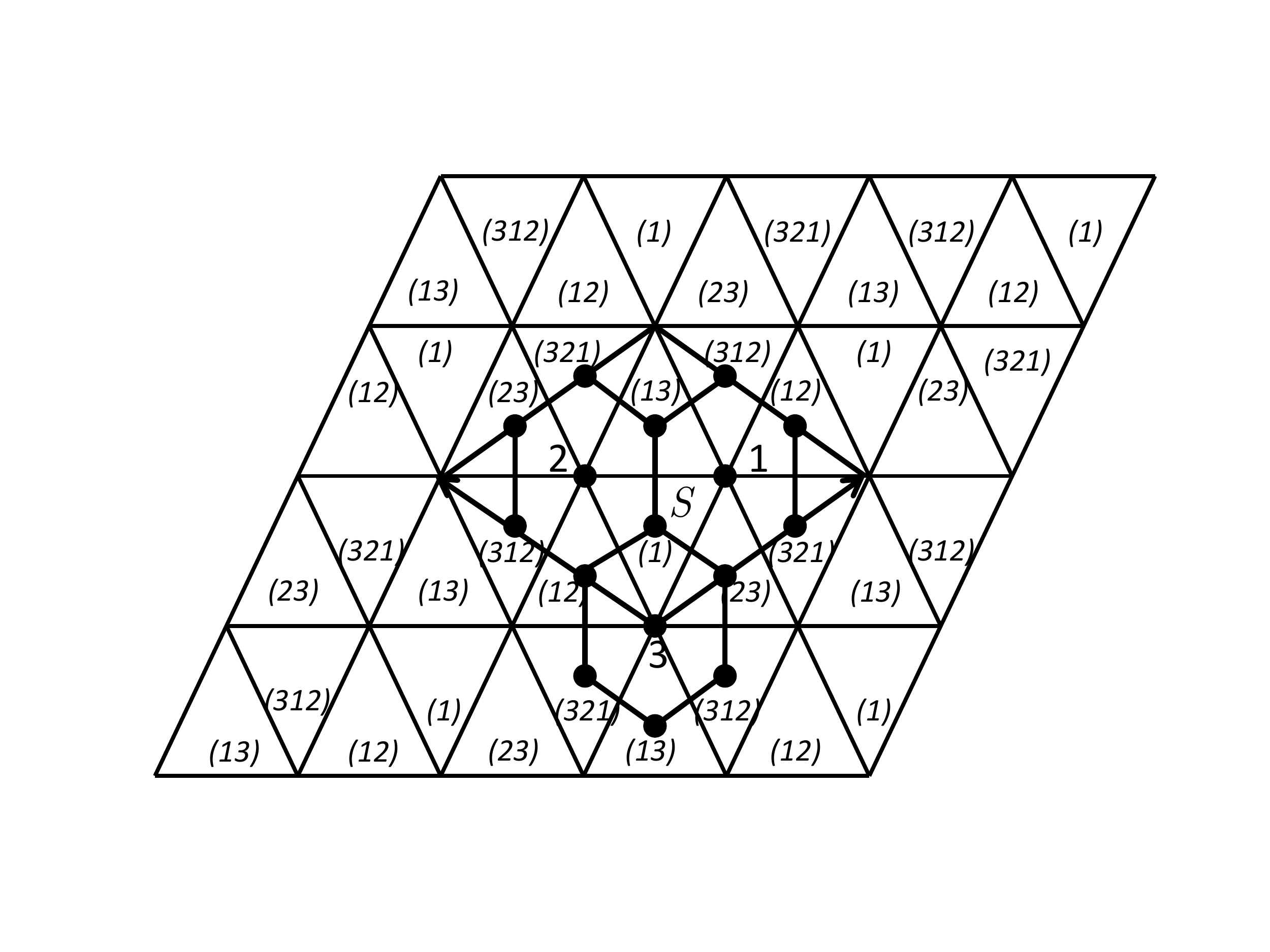}}\\[-.65in]
\caption{\it Shown is the graphicahedron $\mathcal{P}_{C_3}$ for the $3$-cycle $C_3$, viewed as a tessellation of the $2$-torus $\mathbb{E}^{2}/A_{2}$ with three hexagonal ($2$-permutahedral) tiles. The three tiles are the Voronoi regions in $\mathbb{E}^2$ for the lattice $A_{2}^*$, taken at the vertices (labeled $1$, $2$ and $3$) of the fundamental triangle $S$ and considered modulo $A_2$.  The triangles in the Coxeter complex $\mathcal{C}$ of $\tilde{A}_2$ are marked in an $A_2$-invariant manner by permutations in $S_3$. The vertices of the Voronoi tessellation $\mathcal{V}$ for $A_{2}^*$ are the centers of the triangles of $\mathcal{C}$ and can be labeled with the marks of the triangles of which they are the center. In this way, the vertices of the hexagonal tiles of the torus tessellation receive as marks the permutations $\alpha$ of $S_3$ that correspond to them in their representation $(\emptyset,\alpha)$ as vertices of the graphicahedron $\mathcal{P}_{C_3}$.} \label{A2tilde2}
\end{center}
\end{figure}
\medskip

In order to describe the general case it is convenient to let $\tilde{A}_{q-1}$ and its reflection generators $r_1,\ldots,r_q$ act on the $(q-1)$-dimensional linear hyperplane 
\[ E^{q-1}:= \{(x_1,\ldots,x_q) \mid x_{1}+\ldots+x_{q}=0 \} \]
of Euclidean $q$-space $\mathbb{E}^q$. In particular, we take the mirror of $r_i$ to be the intersection of $E^{q-1}$ with the hyperplane of $\mathbb{E}^q$ defined by the equation $x_{i}=x_{i+1}$ if $i<q$, or $x_{1}=x_{q}-1$ if $i=q$ (see 
\cite[p. 456]{cs}). The Coxeter subgroup (of type) $A_{q-1}$ generated by $r_1,\ldots,r_{q-1}$ is the symmetric group $S_q$ (permuting the $q$ axes of $\mathbb{E}^q$ when considered on $\mathbb{E}^q$). This subgroup $A_{q-1}$ naturally has a root system (of type) $A_{q-1}$ associated with it, consisting of the root vectors $\pm (b_{j}-b_{k})$ of $E^{q-1}$ with $1\leq j<k\leq q$, where here $b_1,\ldots,b_q$ denote the canonical basis vectors of $\mathbb{E}^q$; the root vectors are precisely the normal vectors of length $\sqrt{2}$ of the mirrors of the reflections in $A_{q-1}$. The corresponding root lattice (of type) $A_{q-1}$ is the lattice in $E^{q-1}$ spanned by the roots, with a lattice basis given by the (fundamental root) vectors 
\[ a_{i}:=b_{i+1}-b_{i} \;\; (i=1,\ldots,q-1). \] 
If we also set 
\[ a_{q}:=-(a_{1}+\ldots +a_{q-1})=b_{q}-b_{1},\] 
then $a_1,\ldots,a_q$ are the outer normal vectors of the walls of the fundamental region $S$ of the infinite Coxeter group $\tilde{A}_{q-1}$ on $E^{q-1}$. The vertex $v_i$ of $S$ with label $i$ (opposite to the mirror of $r_i$) is given by
\begin{equation}
\label{thevi}
v_{i}:= {\textstyle \frac{i}{q}}(b_{1}+\ldots+b_{q}) -(b_{1}+\ldots+b_{i}) \quad (i=1,\ldots,q) ,
\end{equation}
so in particular, $v_q=o$, the origin.

The infinite group $\tilde{A}_{q-1}$ is obtained from the finite group $A_{q-1}$ by adjoining the translations by root vectors from the root system $A_{q-1}$; more precisely, the translation subgroup of $\tilde{A}_{q-1}$ consists of the translations by vectors of the root lattice $A_{q-1}$. The element $r_{1}r_{2}\ldots r_{q-1}r_{q}r_{q-1}\ldots r_{2}$ of $\tilde{A}_{q-1}$ (corresponding to the left side of (\ref{extra})) is just the translation by $-a_1$; the translation vectors of its conjugates by elements of $A_{q-1}$ generate the root lattice.

The dual lattice $A_{q-1}^*$ of the root lattice $A_{q-1}$ consists of the vectors $v$ in $E^{q-1}$ whose inner products $\langle u,v\rangle$ with the vectors $u$ in the root lattice $A_{q-1}$ are integers (see Figure~\ref{A2tilde1} for an illustration of the case $q=3$).  A lattice basis of $A_{q-1}^*$ is given by the vectors
\[ a_{1},\,a_{1}+a_{2},\,\ldots,\,a_{1}+\! \ldots\! +a_{q-1},\, v_{1}, \]
with $v_1$ as in (\ref{thevi}) above (see \cite[p. 115]{cs}); an alternative basis is $a_{1},\ldots,a_{q-1},v_{1}$. The original lattice $A_{q-1}$ is a sublattice of $A_{q-1}^*$ of index $q$. 

For us it is important to note that the vectors in $A_{q-1}^*$ are just the vertices of the Coxeter complex $\mathcal{C}$ of the infinite Coxeter group $\tilde{A}_{q-1}$. In fact, the vectors  $v_1,\ldots,v_{q-1}$, which are vertices of the fundamental chamber $S$, form the lattice basis of $A_{q-1}^*$ dual to the basis $a_1,\ldots,a_{q-1}$ of the root lattice $A_{q-1}$; note here that when  $i,j\leq q-1$ we have $\langle v_{i},a_{j}\rangle =1$ or $0$ according as $j=i$ or $j\neq i$ (recall that $a_1,\ldots,a_{q-1}$ are the outer normal vectors of the walls of $S$ that contain $o$). Now, since the dual lattice $A_{q-1}^*$ is invariant under the Coxeter group $\tilde{A}_{q-1}$ and this group acts transitively on the chambers of $\mathcal{C}$, it follows in particular that all the vertices of $\mathcal{C}$ must belong to $A_{q-1}^*$. Conversely, suppose there is a vector $v$ of $A_{q-1}^*$ which is not a vertex of $\mathcal{C}$. By replacing $v$ by its image under an element of $\tilde{A}_{q-1}$ (if need be), we may assume that $v$ lies in $S$ but is not a vertex of $S$; in particular, $v\neq o$. But then there must be $q-2$ vectors $v_{j_1},\ldots,v_{j_{q-2}}$ (say) among $v_1,\ldots,v_{q-1}$ which, together with $v$, span a sublattice of $A_{q-1}^*$ whose determinant is smaller than the determinant of $A_{q-1}^*$ (and hence has a fundamental parallelepiped with smaller volume than the fundamental parallelepiped for $A_{q-1}^*$). However, this is impossible, since a sublattice of a lattice cannot have a smaller determinant than the lattice itself. Thus the dual lattice $A_{q-1}^*$ is just the vertex-set of the Coxeter complex $\mathcal{C}$ of the Coxeter group $\tilde{A}_{q-1}$.

The Voronoi cells (regions) of the dual root lattice $A_{q-1}^*$ are the tiles in the Voronoi tessellation $\mathcal{V}=\mathcal{V}(A_{q-1}^*)$ of $E^{q-1}$ for $A_{q-1}^*$ and are known to be $(q-1)$-dimensional permutahedra; see \cite[p. 115]{cs}. (Recall that the Voronoi cell associated with a point $u$ of a lattice consists of all points of the ambient space whose distance from $u$ is shorter than, or equal to, the distance from any other point in the lattice. The Voronoi cells are the tiles in the Voronoi tessellation of the lattice.)  This tessellation $\mathcal{V}$ is invariant under $\tilde{A}_{q-1}$ (since $A_{q-1}^*$ itself is invariant), and in particular, under the translations by the vectors from the root lattice $A_{q-1}$. Moreover, the lattice $A_{q-1}^*$ acts tile-transitively on $\mathcal{V}$ by translations. The $q$ permutahedra centered at the vertices of a chamber of $\mathcal{C}$ share the center of this chamber as a common vertex, and are the only tiles of $\mathcal{V}$ that meet at this vertex. The vertex-set of $\mathcal{V}$ is precisely the set of centers of chambers of $\mathcal{C}$, and the vertex-figures of $\mathcal{V}$ at its vertices are $(q-1)$-simplices. Thus $\mathcal{V}$ is a tiling of $E^{q-1}$ by $(q-1)$-dimensional permutahedra, each labeled with the label (from among $1,\ldots,q$) of the vertex of $\mathcal{C}$ that is its center, and each with its vertices labeled with the labels (from $S_q$, via $\pi$) of the chambers of $\mathcal{C}$ of which they are the center. 

The following theorem describes the structure of the graphicahedron for $C_q$. See Figure~\ref{A2tilde2} for an illustration of the case $q=3$, when $\mathcal{P}_{C_3}$ is just the regular toroidal map $\{6,3\}_{(1,1)}$ (see \cite{graphi}). When $q\geq 4$ the graphicahedron $\mathcal{P}_{C_q}$ is not a regular polytope.

\begin{theorem}
\label{voronoi}
Let $q\geq 3$. Suppose the Coxeter group $\tilde{A}_{q-1}$, the root lattice $A_{q-1}$, the dual root lattice $A_{q-1}^*$, and the Voronoi tiling $\mathcal{V}=\mathcal{V}(A_{q-1}^*)$ for $A_{q-1}^*$ in the Euclidean $(q-1)$-space $E^{q-1}$ are defined as described above. Then the quotient tessellation $\mathcal{V}/A_{q-1}$ of $\mathcal{V}$ by $A_{q-1}$ is a face-to-face tessellation of the $(q-1)$-torus $E^{q-1}/A_{q-1}$ by $(q-1)$-dimensional permutahedra, with $q$ such tiles, and the face poset of $\mathcal{V}/A_{q-1}$ is isomorphic to the $C_q$-graphicahedron~$\mathcal{P}_{C_q}$.
\end{theorem} 
 
\begin{proof}
Since the tessellation $\mathcal{V}$ in $E^{q-1}$ by permutahedra is invariant under the translations by the vectors from the root lattice $A_{q-1}$, it is immediately clear that the quotient $\mathcal{V}/A_{q-1}$ is a face-to-face tessellation of the $(q-1)$-torus $E^{q-1}/A_{q-1}$, with tiles that are again permutahedra. Since $A_{q-1}^*$ contains $A_{q-1}$ as a sublattice of index $q$ and the tiles of $\mathcal{V}$ are centered at the points of $A_{q-1}^*$, the tessellation $\mathcal{V}/A_{q-1}$ on $E^{q-1}/A_{q-1}$ has exactly $q$ tiles, centered (in a sense) at the points of $A_{q-1}^{*}/A_{q-1}$. This completes the first part of the theorem.

Note that the number of flags of $\mathcal{V}/A_{q-1}$ is $(q!)^2$, which is also the number of flags of $\mathcal{P}_{C_q}$. In fact, there are $q$ tiles in $\mathcal{V}/A_{q-1}$, each a permutahedron with $q!$ vertices; moreover, each vertex of a tile has a vertex-figure which is a $(q-2)$-simplex. This gives a total of $q\cdot (q!)\cdot (q-1)! = (q!)^2$ flags. 

For the second part of the theorem we employ the epimorphism $\pi:\tilde{A}_{q-1}\rightarrow S_q$ of (\ref{maptau}) to construct an isomorphism between the face posets of $\mathcal{V}/A_{q-1}$ and $\mathcal{P}_{C_q}$. We first describe  an epimorphism (incidence preserving surjective mapping) from the face poset of $\mathcal{V}$ to the face poset of $\mathcal{P}_{C_q}$ which maps any two faces of $\mathcal{V}$ equivalent under the root lattice $A_{q-1}$ to the same face of $\mathcal{P}_{C_q}$, and which is a local isomorphism, meaning that it maps the face poset of every tile of $\mathcal{V}$ isomorphically to the face poset of a facet of $\mathcal{P}_{C_q}$. 

To this end, recall that every vector $u$ of $A_{q-1}^*$ is the center of a tile $P_u$ of $\mathcal{V}$; and vice versa, that all tiles of $\mathcal{V}$ are obtained in this way from vectors in $A_{q-1}^*$. As a vertex of the Coxeter complex $\mathcal{C}$ of $\tilde{A}_{q-1}$, the vector $u$ naturally carries a label $i_u$ (say) from among $1,\ldots,q$. If $u$ happens to be a vertex of the fundamental chamber $S$ of $\mathcal{C}$, then the vertex-stabilizer of $u$ in the Coxeter complex is the finite Coxeter subgroup $\langle r_j\mid j\neq i_u\rangle$ of type $A_{q-1}$ isomorphic to $\langle \tau_j\mid j\neq i_u\rangle = \langle \tau_{e_j}\mid j\neq i_u\rangle = S_q$. Thus in this case the permutahedron $P_u$ is naturally associated with this Coxeter subgroup; in fact, $P_u$ is the graphicahedron for this subgroup, or rather for the path $C_q^{i_u}$ of length $q-1$ obtained from $C_q$ by deleting the edge $e_{i_u}$ from $C_q$. If $u$ is an arbitrary vertex of $\mathcal{C}$, then $u$ is equivalent under $\tilde{A}_{q-1}$ to a unique vertex of $S$, namely the vertex $u'$ (say) of $S$ with the same label, $i_u$. In this case $P_u$ still is associated with the Coxeter subgroup $\langle r_j\mid j\neq i_u\rangle$ of $\tilde{A}_{q-1}$ (or rather a subgroup of $\tilde{A}_{q-1}$ conjugate to it); hence, since $i_{u}=i_{u'}$ and $P_u$ is equivalent under $\tilde{A}_{q-1}$ to $P_{u'}$, the tile 
$P_u$ still is the graphicahedron for $C_q^{i_u}$.   

The construction of the epimorphism from the face poset of $\mathcal{V}$ to the face poset of $\mathcal{P}_{C_q}$ proceeds in two steps:\ we first define its effect locally on the face lattices of the tiles $P_u$, and then we establish that these ``local isomorphisms" fit together globally to give the desired epimorphism. 

To begin with, consider a tile $P_u$ of $\mathcal{V}$ for some $u$ in $A_{q-1}^*$. Identifying the three groups 
\[ \langle r_j\mid j\neq i_u\rangle,\;\; 
\langle \tau_{j}\mid j\neq i_u\rangle,\;\;
\langle \tau_{e_j}\mid j\neq i_u\rangle\] 
with each other in an obvious way, we can express (up to equivalence under $A_{q-1}^*$) a typical face of $P_u$ in the form $(K,\alpha)$, where $K\subseteq \{j\mid j\neq i_u\}$ and $\alpha\in\langle r_j\mid j\neq i_u\rangle$. If $(K,\alpha)$ and $(L,\beta)$ are two faces of $P_u$ with $K,L\subseteq \{j\mid j\neq i_u\}$ and $\alpha,\beta\in\langle r_j\mid j\neq i_u\rangle$, then $(K,\alpha)\leq (L,\beta)$ in $P_u$ if and only if $K\subseteq L$ and $T_{K}\alpha=T_{L}\beta$ (with $T_{K}=\langle r_j \mid j \in K\rangle$ and $T_{L}=\langle r_j \mid j \in L\rangle$). We now employ the group epimorphism $\pi:\tilde{A}_{q-1}\rightarrow S_q$ of (\ref{maptau}), or more exactly, its restriction to the subgroup $\langle r_j\mid j\neq i_u\rangle$. Consider the mapping from the face poset of $P_u$ to the face poset of $\mathcal{P}_{C_q}$ defined by 
\begin{equation}
\label{lociso}
(K,\alpha) \rightarrow (\widehat{K},\pi(\alpha)) 
\;\quad (K\subseteq \{j\mid j\neq i_u\},\, \alpha\in\langle r_j\mid j\neq i_u\rangle),
\end{equation}
where 
\[ \widehat{K}:=\{e_k\mid k\in K\} .\] 
Now, if 
\[ (K,\alpha)\leq (L,\beta) \] 
in $P_u$, then $K\subseteq L$ and $T_{K}\alpha=T_{L}\beta$, so also $\widehat{K}\subseteq \widehat{L}$ and $T_{\widehat{K}}\,\pi(\alpha)=T_{\widehat{L}}\,\pi(\beta)$ (here with $T_{\widehat{K}}=\langle \tau_{e} \mid e \in \widehat{K}\rangle$ and $T_{\widehat{L}}=\langle \tau_{e} \mid e \in \widehat{L}\rangle$) and therefore 
\[ (\widehat{K},\pi(\alpha))\leq (\widehat{L},\pi(\beta)) \] 
in $\mathcal{P}_{C_q}$. Thus the mapping in (\ref{lociso}) preserves incidence. Moreover, its image in $\mathcal{P}_{C_q}$ consists of all faces of $\mathcal{P}_{C_q}$ of the form $(M,\gamma)$ with 
\[ M\subseteq \{e\mid e\neq e_{i_u}\}=E(C_q^{i_u}), \]
the edge set of $C_q^{i_u}$, and $\gamma\in\langle\tau_e \mid e\neq e_{i_u}\rangle$; hence this image is just the face poset of the permutahedral facet $\mathcal{P}_{C_q^{i_u}}$ of $\mathcal{P}_{C_q}$. In particular, since $\pi$ induces an isomorphism between $\langle r_j\mid j\neq i_u\rangle$ and $\langle \tau_e\mid e\neq e_{i_u}\rangle$, the resulting mapping between the face posets of $P_u$ and $\mathcal{P}_{C_q^{i_u}}$ must actually be an isomorphism.

It remains to show that these local isomorphisms fit together coherently. First recall that the vertices of $\mathcal{V}$ lie at the centers of the chambers of $\mathcal{C}$ and hence can be identified with the elements of the Coxeter group $\tilde{A}_{q-1}$ that are associated with their respective chamber. Now suppose we have a common (non-empty) face of the tiles $P_u$ and $P_v$ of $\mathcal{V}$, where again $u,v\in A_{q-1}^*$. Then $u$ and $v$ must necessarily be adjacent vertices of a chamber of $\mathcal{C}$, and since $\tilde{A}_{q-1}$ acts transitively on the chambers, we may assume that this is the fundamental chamber $S$ of $\mathcal{C}$. The common face of $P_u$ and $P_v$ has two representations, namely as $(K_u,\alpha_u)$, with 
\[ K_u\subseteq \{j\mid j\neq i_u\},\; \alpha_u\in\langle r_j\mid j\neq i_u\rangle ,\]
and as $(K_v,\alpha_v)$, with 
\[ K_v\subseteq \{j\mid j\neq i_v\},\; \alpha_v\in\langle r_j\mid j\neq i_v\rangle. \] 
Our goal is to show that this forces $K_{u}=K_v$ and also allows us to assume that $\alpha_{u}=\alpha_v$. Then clearly this would imply $\widehat{K}_{u}=\widehat{K}_v$ and $\pi(\alpha_u)=\pi(\alpha_v)$, so the local definition of the mapping in (\ref{lociso}) would not depend on the representation of a face of $\mathcal{V}$ as a face of $P(u)$ or $P(v)$, as desired. Note that we already know at this point that $K_{u}$ and $K_v$ have the same cardinality, which is just the dimension (rank) of the face.

First observe that, if need be, we may replace $\alpha_u$ and $\alpha_v$ by any suitable group element for which the vertices $(\emptyset,\alpha_u)$ of $P_u$ and $(\emptyset,\alpha_v)$ of $P_v$ are vertices of the common face $(K_u,\alpha_u)$ of $P_u$ and $(K_v,\alpha_v)$ of $P_v$, respectively. In $E^{q-1}$, these vertices are just the centers of the chambers corresponding to the elements $\alpha_u$ and $\alpha_v$ of $\tilde{A}_{q-1}$. Hence, by choosing the same vertex in the two representations of the common face, we may, and will, assume that $\alpha_u=\alpha_v=:\alpha$. Then also 
\[ \alpha \in \langle r_j\mid j\neq i_u\rangle \cap \langle r_j\mid j\neq i_v\rangle
= \langle r_j\mid j\neq i_u,i_v\rangle, \]
by the intersection property for the distinguished (parabolic) subgroups of $\tilde{A}_{q-1}$ (see \cite[Ch. 3A]{arp}). 

Suppose for a moment that the common face is actually a facet of $P_u$ and $P_v$ in $\mathcal{V}$, so that $P_u$ and $P_v$ are adjacent tiles. Then this facet, and in particular its vertex identified with $\alpha$, must lie in the perpendicular bisector of $u$ and $v$ in $E^{q-1}$. The stabilizer of the edge $\{u,v\}$ of the fundamental chamber $S$ of $\mathcal{C}$ in $\tilde{A}_{q-1}$ is just the subgroup $\langle r_j\mid j \neq i_u,i_v\rangle$ of $\tilde{A}_{q-1}$, which acts irreducibly on the perpendicular bisector. Now 
\[ (\{j\mid j\neq i_u,i_v\},\,\alpha) \] 
is a common facet of $P_u$ and $P_v$ (when viewed in either representation of faces). On the other hand, the vertex-set of this facet must necessarily coincide with the orbit of the vertex $\alpha$ under the edge stabilizer subgroup $\langle r_j\mid j \neq i_u,i_v\rangle$. Thus
\[ K_{u}=K_{v}=\{j \mid j\neq i_u,i_v\} . \] 
This settles the case when the common face is a facet.

Now, to deal with the general case recall that any two tiles of $\mathcal{V}$ centered at vertices of $S$ are adjacent in $\mathcal{V}$. Hence, if a face of dimension $d$ (say) is common to both tiles $P_u$ and $P_v$, then this face is the intersection of precisely $q-d$ (mutually adjacent) tiles $P_{u_1},\ldots,P_{u_{q-d}}$ (say), where 
\[ u=u_1,u_2,\ldots,u_{q-d-1},u_{q-d}=v\] 
are vertices of the same chamber. This gives $q-d$ representations $(K_{u_j}, \alpha_{u_j})$ of this common face, one for each $j$. Now comparing any two such representations as above, we first observe that we may choose the same element $\alpha$ as $\alpha_{u_j}$ for each $j$; in particular, 
\[ \alpha \in \langle r_j\mid j\neq i_{u_{1}},\dots,i_{u_{q-d}}\rangle, \] 
again by the intersection property of $\tilde{A}_{q-1}$. Further, by viewing $(K_{u_k}, \alpha_{u_k})=(K_{u_k}, \alpha)$ as a face of a facet $(L_{u_k},\alpha)$ of $P_{u_k}$ which is shared by another tile $P_{u_l}$ (say) among $P_{u_1},\ldots,P_{u_{q-d}}$, we find that 
\[K_{u_k},K_{u_l} \subseteq L_{u_k}=L_{u_l}=\{j \mid j\neq i_{u_k},i_{u_l}\} .\]
It follows that 
\[ K_{u_k}\subseteq \{j \mid j\neq i_{u_1},\ldots,i_{u_{q-d}}\} .\] 
Finally, since $K_{u_k}$ has cardinality~$d$, we must actually have 
\[ K_{u_k}=\{j \mid j\neq i_{u_1},\ldots,i_{u_{q-d}}\}, \] 
independent of $k$. Thus $K_u=K_v$. 

At this point we have a homomorphism from the face poset of $\mathcal{V}$ to the face poset of $\mathcal{P}_{C_q}$ that is defined by (\ref{lociso}). Clearly, this is an epimorphism, since the face poset of every facet of $\mathcal{P}_{C_q}$ is the image of the face poset of a suitable tile of $\mathcal{V}$ (every label from $1,\ldots,q$ occurs as a label of a vertex of $\mathcal{C}$). Moreover, this epimorphism is invariant under the translations by vectors from the root lattice $A_{q-1}$, as these translations lie in $\tilde{A}_{q-1}$ and the action of $\tilde{A}_{q-1}$ on $\mathcal{C}$ preserves the labels of vertices of $\mathcal{C}$.

Finally, then, by the invariance under $A_{q-1}$, the above epimorphism from the face poset of $\mathcal{V}$ to the face poset of $\mathcal{P}_{C_q}$ also determines an epimorphism from the face poset of $\mathcal{V}/A_{q-1}$ to the face poset of $\mathcal{P}_{C_q}$ that is again a local isomorphism. Under this induced epimorphism, the equivalence class of a face of $\mathcal{V}$ modulo $A_{q-1}$ is mapped to the image of any representative of this equivalence class under the original epimophism. On the other hand, since the face posets of $\mathcal{V}/A_{q-1}$ and $\mathcal{P}_{C_q}$ have the same number of flags, this induced epimorphism must actually be an isomorphism 
between the face posets of $\mathcal{V}/A_{q-1}$ and $\mathcal{P}_{C_q}$. This completes the proof.
\end{proof} 

Theorem~\ref{voronoi} rests on the geometry of the Euclidean Coxeter group $\tilde{A}_{q-1}$ and on the observation that the symmetric group $S_q$, via (\ref{maptau}), is an epimorphic image of $\tilde{A}_{q-1}$. This quotient relationship between $\tilde{A}_{q-1}$ and $S_q$ can be seen as a degenerate case, obtained when $s=1$, among similar quotient relationships between $\tilde{A}_{q-1}$ and the groups in a certain family of unitary reflection groups. In fact, when the single extra relation
\begin{equation}
\label{extrarelins}
(r_{1}r_{2}\ldots r_{q-1}r_{q}r_{q-1}\ldots r_{2})^{s}=1,
\end{equation}
for any $s\geq 2$, is added to the Coxeter relations of $\tilde{A}_{q-1}$, we obtain a presentation for the finite unitary group $W_{q}(s):=[1\;1\;\ldots 1]^s$ (with $q$ entries $1$) generated by involutory unitary reflections in unitary complex $q$-space (see \cite[9A17]{arp}). Following \cite{coxunitary}, a convenient circular diagram representation for $W_{q}(s)$ is obtained by adding a mark $s$ inside the standard circular Coxeter diagram for $\tilde{A}_{q-1}$ given by $C_q$ (see \cite[Figure 9A4]{arp}).  As explained in \cite[9A16, 9A17]{arp}, the same group $W_{q}(s)$ also has an alternative representation (based on new generators $r'_1,\ldots,r'_q$) by a diagram consisting of a triangle with a tail, in which the triangle, with nodes $1,2,3$ (say), has a mark $3$ inside and the branch $\{1,2\}$ opposite the tail is marked~$s$. This new diagram encodes an alternative presentation for $W_{q}(s)$ consisting of the standard Coxeter relations for the (new) underlying Coxeter diagram and the single extra relation 
\[ (r'_{1}r'_{2}r'_{3}r'_{2})^{3} = 1 . \]
(Of course, we could avoid complex reflection groups altogether by directly rewriting $\tilde{A}_{q-1}$ as a group represented by a triangle with a tail, now with a mark $\infty$ on branch $\{1,2\}$.) Now allow the degenerate case $s=1$. Then the same change of generators that transformed the circular diagram of $W_{q}(s)$ to the diagram given by a triangle with a tail works in this case as well. On the other hand, when $s=1$ the mark $s$ on the branch $\{1,2\}$ identifies the generators $s_1$ and $s_2$ (there is no impact from the mark inside the triangle or from other Coxeter relations), to give the desired conclusion that 
\[ W_{q}(1)= \langle r'_{2},\ldots,r'_{q}\rangle = S_q .\]
Thus the circular diagram with a mark $1$ inside also represents the symmetric group $S_q$. 

Now, since the original symmetric group $S_q$ with its generating transpositions $\tau_1,\ldots,\tau_q$ also satisfies the relations represented by the circular diagram with a mark $1$ inside (that is, relations (\ref{cox}) and (\ref{extra})), it must actually be isomorphic to the group abstractly defined by these relations. Thus (\ref{cox}) and (\ref{extrarelins}), with $s=1$, abstractly define $S_q$. 
\bigskip

Concluding we remark that the graphicahedron $\mathcal{P}_{C_4}$ (of rank $3$) for the $4$-cycle $C_4$ also permits an alternative representation based on a different description of the corresponding Coxeter group $\tilde{A}_3$. In fact, the $3$-dimensional Coxeter group $\tilde{A}_3$ can also be viewed as a subgroup of index $4$ in the symmetry group $[4,3,4]$ of a regular cubical tessellation $\{4,3,4\}$ in $\mathbb{E}^3$, which also is a discrete reflection (Coxeter) group in $\mathbb{E}^3$. Recall that the Coxeter complex for $[4,3,4]$ is built from the barycentric subdivisions of the cubical tiles of $\{4,3,4\}$. In this Coxeter complex, the four tetrahedra that surround a common edge of type $\{1,2\}$ fit together to form a larger tetrahedron. The four reflections in the walls of any such larger  tetrahedron $S$ (say) can be seen to generate a discrete reflection group in $\mathbb{E}^3$ with $S$ as its fundamental tetrahedron. This group is a representation of $\tilde{A}_3$ as a reflection group on $\mathbb{E}^3$. Its index as a subgroup of $[4,3,4]$ is $4$, since $S$ is dissected into $4$ smaller tetrahedra of the Coxeter complex for $[4,3,4]$. For figures illustrating the structure of the graphicahedron $\mathcal{P}_{C_4}$ as a tessellation of a $3$-torus with $4$ permutahedral tiles arising from this  alternative description of $\tilde{A}_3$, see \cite{delrio}.
\smallskip

\noindent\textbf{Acknowledgements}.
We are grateful to the referee for a careful reading of the original manuscript and various helpful comments.

\end{document}